\documentclass[11pt]{article}

\usepackage{amsmath,amsfonts,amssymb,amsthm}
\usepackage{graphicx}
\graphicspath{{figures/}}
\usepackage{float}
\usepackage{caption}
\usepackage{booktabs}
\usepackage{authblk}
\usepackage{etoolbox}
\usepackage[top=2.5cm,bottom=2.5cm,left=3cm,right=3cm]{geometry}
\usepackage[hidelinks]{hyperref}
\hypersetup{
  pdftitle={Convergence of Least Squares Methods for PDEs on R\textasciicircum d with Shallow ReLU\textasciicircum k Networks},
  pdfauthor={Tong Mao, Lei Shi, and Hanfei Zhou},
  pdfkeywords={shallow ReLU\textasciicircum k networks, weighted least squares, Gaussian Sobolev spaces, inverse inequalities}
}

\apptocmd{\thebibliography}{\setlength{\itemsep}{0pt}}{}{}
\allowdisplaybreaks[4]

\newcommand{\R}{\mathbb R}
\newcommand{\N}{\mathbb N}

\newcommand{\dd}{\,\mathrm d}
\newcommand{\Sph}{\mathbb S}

\newcommand{\cA}{\mathcal A}

\newcommand{\cH}{\mathcal H}

\newcommand{\norm}[1]{\left\|#1\right\|}

\newcommand{\floor}[1]{\left\lfloor #1\right\rfloor}

\newtheorem{theorem}{Theorem}
\newtheorem{assumption}{Assumption}
\newtheorem{corollary}{Corollary}

\newtheorem{lemma}{Lemma}
\newtheorem{proposition}{Proposition}

\numberwithin{equation}{section}

\title{Convergence of Least Squares Methods for PDEs on $\mathbb R^d$ with Shallow $\mathrm{ReLU}^k$ Networks$^\dag$\footnotetext{\dag~The work of Lei Shi is partially supported by the National Natural Science Foundation of China (Grant No.~12571099). All authors contributed equally to this work and are listed alphabetically. The corresponding author is Hanfei Zhou.}}

\author[1]{Tong Mao}
\author[2,3]{Lei Shi}
\author[2]{Hanfei Zhou}

\affil[1]{Shenzhen Loop Area Institute, Shenzhen 518000, China}
\affil[2]{School of Mathematical Sciences, \linebreak
Fudan University, Shanghai 200433, China
}
\affil[3]{
Shanghai Key Laboratory for Contemporary Applied Mathematics, \linebreak
Fudan University, Shanghai 200433, China
}
\affil[ ]{Email: tongmao@slai.edu.cn, leishi@fudan.edu.cn, zhouhf23@m.fudan.edu.cn}
\date{}

\begin{document}
	\maketitle
 
\begin{abstract}
We establish inverse inequalities in Gaussian Sobolev norms for shallow $\mathrm{ReLU}^k$ trial spaces. The estimates identify the sharp regularity threshold and quantify the dependence on parameter separation and bias magnitude. Based on these estimates, we prove convergence of weighted least squares discretizations of linear PDEs on $\mathbb R^d$ using shallow networks with prescribed inner parameters and finitely many residual samples in an expanding ball. Under suitable stability and regularity assumptions, we obtain an error bound of order $h^{k+(d+1)/2-r}\delta_\pm^{-(S-r)}$ in the Gaussian Sobolev norm of order $S$, up to logarithmic factors, where $r$ is the PDE order, $h$ is the parameter fill distance, and $\delta_\pm$ is the antipodal separation. The convergence results cover deterministic and random sampling and imply error bounds on fixed bounded regions by restriction. We construct parameter sets with $\delta_\pm\asymp h$, yielding a convergence rate in terms of the network width that matches a lower bound up to logarithmic factors. Numerical experiments support the theoretical results and provide practical guidance on balancing accuracy and stability.
\end{abstract}


\section{Introduction}

Residual minimization is widely used in neural network methods for solving partial differential equations (PDEs) \cite{KarniadakisEtAl2021PhysicsInformed}, and the resulting accuracy depends on approximation, sampling, and optimization \cite{DeRyckMishra2024,MishraMolinaro2023,WangTengPerdikaris2021}. When all parameters are trained, residual minimization is generally nonconvex, and the resulting optimization error can be difficult to quantify \cite[Sections~2.5 and~7.2]{DeRyckMishra2024}. To avoid this nonconvex optimization, one can fix the inner parameters and determine only the outer coefficients, yielding a convex least squares problem for linear PDEs. This principle underlies extreme learning machines \cite{HuangZhuSiew2006} and has been used in neural collocation and related PDE solvers \cite{CalabroFabianiSiettos2021,ChiChenYang2024,DongLi2021}.

Shallow neural networks are commonly used in these methods. For the $\mathrm{ReLU}^k$ activation, the trial functions on $\mathbb R^d$ take the form \[ v(x)=\sum_{j=1}^J c_j(a_j\cdot x-b_j)_+^k, \qquad a_j\in\mathbb S^{d-1},\quad b_j\in\mathbb R. \] Here $k\ge1$ is an integer, $J$ is the network width, and $z_+:=\max\{z,0\}$. The approximation properties of shallow networks have been studied \cite{SiegelXu2022HighOrder,MaoZhou2023,SiegelXu2024Sharp,MaoSiegelXu2026,LiuMaoXu2025Integral,Xu2020FiniteNeuron}.
These approximation estimates do not by themselves determine the accuracy of a discrete least squares solution. Recovering the approximation order requires discrete stability \cite{NarayanJakemanZhou2017,CohenDavenportLeviatan2013,CohenMigliorati2017}, which we obtain by combining PDE stability with sampling inequalities \cite{NarcowichWardWendland2005,WendlandRieger2005} and inverse inequalities \cite{BrennerScott2008,MhaskarNarcowichPrestinWard2010}, which for kernel spaces depend on the separation of the centers \cite{NarcowichWardWendland2006,NarcowichSunWardWendland2007,Ward2012}. For shallow $\mathrm{ReLU}^k$ networks, these inverse estimates are sensitive to the geometry of the prescribed directions and biases, since nearly coincident $\mathrm{ReLU}^k$ neurons can produce cancellations. The central difficulty is to quantify the dependence on parameter separation and bias magnitude as the trial space grows.

PDEs on $\mathbb R^d$ include Fokker--Planck equations in stochastic dynamics \cite{DujardinHerauLafitte2020,FokGuoTang2002}. As in Hermite spectral methods \cite{FunaroKavian1991,Guo1999Hermite}, we work in Gaussian Sobolev spaces $\mathcal H_\gamma^s$ associated with the standard Gaussian measure $\gamma$; neural network approximation in $L^2(\mathbb R^d,\gamma)$ is studied in \cite{SchwabZech2023}. For shallow $\mathrm{ReLU}^k$ networks with fixed inner parameters, integral representations and optimal approximation rates on bounded domains are established in \cite{LiuMaoXu2025Integral}, and sharp condition-number estimates on $\mathbb S^d$ for antipodally quasiuniform parameters are given in \cite{LiuMaoXu2025Stability}; residual least squares convergence for spectral equations on $\mathbb S^d$ is proved in \cite{LiuMaoXu2026DeterministicLS} via a Bernstein inequality in terms of the antipodal separation. On $\mathbb R^d$, the bias range grows with the sampling region, the operator may have unbounded coefficients, and the global error must be controlled from residual samples in a bounded region. We handle these issues with Gaussian inverse estimates, local residual estimates that account for the coefficient growth, and a coefficient constraint that controls the unsampled tails. Because the analysis works directly with affine networks and residual samples in $\mathbb R^d$, it also gives error bounds on fixed bounded regions by restriction, for targets that extend to solutions on $\mathbb R^d$. By contrast, the bounded-domain result in \cite{LiuMaoXu2026DeterministicLS} covers only the order-zero case, that is, function approximation, and requires lifting the target to the sphere with a parity extension.

The main contributions of this work are as follows.
\begin{itemize}

\item We establish the inverse inequality \[ \|v\|_{\mathcal H_\gamma^\tau} \le C\max\{\delta_\pm^{-1},1+R\}^{\tau-\sigma} \|v\|_{\mathcal H_\gamma^\sigma}, \qquad 0\le\sigma\le\tau<k+1/2, \] for functions in shallow $\mathrm{ReLU}^k$ trial spaces. Here $\delta_\pm>0$ is the minimum parameter separation modulo antipodal identification, and $R$ controls the bias range. The constant $C$ is independent of $R$ and the network width. We obtain this estimate through harmonic analysis based on Hermite expansions \cite{PetrushevXu2008}. The polynomial dependence on the bias range allows us to enlarge the sampling region for Gaussian tail control with only logarithmic losses in the convergence rate. We also establish sharpness results for the regularity threshold and the dependence on parameter separation and bias magnitude.

\item Building on this inverse inequality, we prove convergence of weighted least squares approximations to linear PDEs on $\mathbb R^d$, using finitely many residual samples in an expanding ball. A coefficient constraint provides uniform Gaussian tail control. Under the stated assumptions, the weighted least squares approximation
$u_h$ satisfies
\[
    \|u-u_h\|_{\mathcal H_\gamma^S}
    \le
    C\bigl(\log(e/h)\bigr)^\chi
    h^{k+(d+1)/2-r}\,\delta_\pm^{-(S-r)},
\]
where $u$ is the exact solution, $h$ measures parameter coverage,
$r$ is the differential order of the PDE, and $\chi>0$ is a fixed
exponent. We construct parameter sets with $\delta_\pm\asymp h$,
yielding the convergence rate $J^{-(k+(d+1)/2-S)/d}$
up to logarithmic factors. The results cover deterministic and random residual sampling and give convergence on fixed bounded regions by restriction under suitable extension assumptions.

\item We prove a matching lower bound for a compactly supported Sobolev class, valid for every linear trial space of dimension at most $J$. This shows that the convergence exponent cannot be improved uniformly over this class by changing the prescribed inner parameters or the procedure used to compute the outer coefficients.

\item Numerical experiments assess convergence in Sobolev norms and reproduce the predicted growth of inverse constants as the inner parameters approach one another. They distinguish instability caused by insufficient residual sampling from linear dependence caused by antipodal configurations, which persists under oversampling.

\end{itemize}

The paper is organized as follows. In \autoref{sec:preliminaries}, we introduce the function spaces and network parameter geometry. We state the main results in \autoref{sec:main-results} and present numerical experiments in \autoref{sec:experiments}. The main results are proved in \autoref{sec:proofs}. Appendices~\ref{app:inverse-lemmas} and~\ref{app:convergence-lemmas} collect auxiliary results for the inverse inequalities and convergence analysis, respectively.

\section{Preliminaries}
\label{sec:preliminaries}

In this section, we introduce the function spaces, shallow $\mathrm{ReLU}^k$ networks, and parameter distances used in the inverse estimates.

Let $d\gamma(x)=\rho_d(x)\,dx$, where $\rho_d(x)=(2\pi)^{-d/2}e^{-|x|^2/2}$. In one dimension, write $d\gamma_1(z)=\rho_1(z)\,dz$. The normalized probabilists' Hermite polynomials \cite{Szego1975} are defined by
\[
  e^{tz-t^2/2}=\sum_{n=0}^{\infty}H_n(z)\frac{t^n}{\sqrt{n!}}.
\]
$\{H_n\}_{n\ge0}$ is an orthonormal basis of $L^2(\gamma_1)$. The tensor products $H_\alpha(x):=\prod_{\ell=1}^d H_{\alpha_\ell}(x_\ell)$, $\alpha\in\mathbb N_0^d$, form an orthonormal basis of $L^2(\gamma)$. For $f=\sum_\alpha\widehat f_\alpha H_\alpha$, define
\begin{equation}\label{eq:Hermite-Sobolev-definition} \|f\|_{\mathcal H_\gamma^s}^2 := \sum_{\alpha\in\mathbb N_0^d} (1+|\alpha|)^s|\widehat f_\alpha|^2, \qquad s\ge0. \end{equation}
For integers $m\ge0$,
\begin{equation} \label{eq:Hermite-derivative-norm-equivalence} \|f\|_{\mathcal H_\gamma^m} \asymp \left( \sum_{|\alpha|\le m} \int_{\mathbb R^d}|D^\alpha f(x)|^2\,d\gamma(x) \right)^{1/2}. \end{equation}
We write $L_\gamma^2(E)=L^2(E,\gamma)$ and $\|f\|_{H_\gamma^m(E)}^2:=\sum_{|\alpha|\le m}\int_E|D^\alpha f(x)|^2\,d\gamma(x)$ for measurable $E\subset\mathbb R^d$. In these Gaussian norms, an omitted domain means $\mathbb R^d$. Write $B_R:=\{x\in\mathbb R^d:|x|<R\}$. The unweighted space $H^t(B_R)$, $t\ge0$, carries the restriction norm induced by $H^t(\mathbb R^d)$, equivalent to the usual derivative norm for integer $t$.

For an integer $k\ge1$, set $\phi_{a,b}(x):=(a\cdot x-b)_+^k$, $(a,b)\in\mathbb S^{d-1}\times\mathbb R$. A neuron $(w\cdot x+\beta)_+^k$ with $w\neq0$ takes this form with $a=w/|w|$, $b=-\beta/|w|$, and the factor $|w|^k$ absorbed into its outer coefficient. Fix $C_b>2$ and $R\ge1$, and set $\Theta_R:=\mathbb S^{d-1}\times[-C_b(1+R),C_b(1+R)]$.
For $\theta=(a,b)$ and $\theta'=(a',b')$ in $\Theta_R$, let
\begin{equation} \label{eq:parameter-metrics} \begin{aligned} d_0(\theta,\theta') &:= \rho_{\mathbb S^{d-1}}(a,a')+|b-b'|,\qquad d_\pi(\theta,\theta') := \rho_{\mathbb S^{d-1}}(a,-a')+|b+b'|,\\ d_\pm(\theta,\theta') &:= \min\{d_0(\theta,\theta'),d_\pi(\theta,\theta')\}. \end{aligned} \end{equation}
Here $\rho_{\mathbb S^{d-1}}(a,a'):=\arccos(a\cdot a')$; this agrees with geodesic distance for $d\ge2$ and assigns distance $\pi$ to the two points of $\mathbb S^0$. For a finite dictionary $\Xi=\{\theta_j\}_{j=1}^J\subset\Theta_R$, set $V_\Xi:=\operatorname{span}\{\phi_\theta:\theta\in\Xi\}$ and define
\begin{equation} \label{eq:parameter-geometries} \delta(\Xi):=\min_{i\ne j}d_0(\theta_i,\theta_j), \qquad \delta_\pm(\Xi):=\min_{i\ne j}d_\pm(\theta_i,\theta_j). \end{equation}
These are the ordinary and antipodal separation distances, computed with the unscaled bias. Both are $+\infty$ when $\#\Xi\le1$. When the dictionary is clear, we abbreviate $\delta_\pm(\Xi)$ to $\delta_\pm$.

In estimates, $c,C>0$ may change between occurrences.

\section{Main Results}
\label{sec:main-results}

In this section, we state the inverse inequality, the resulting convergence estimates, and a matching lower bound for fixed linear trial spaces.

\subsection{Sharp inverse inequalities}

\begin{theorem}
\label{thm:inverse}
Let $R\ge1$ and $\Xi\subset\Theta_R$ with $\delta_\pm(\Xi)>0$. Then, for any real numbers $0\le\sigma\le\tau<k+1/2$, there exists a constant $C$, depending only on $d,k,\sigma,\tau,C_b$, such that
\begin{equation}
\label{eq:scale-inverse}
  \|v\|_{\mathcal H_\gamma^\tau}
  \le
  C\max\{\delta_\pm^{-1},1+R\}^{\tau-\sigma}
  \|v\|_{\mathcal H_\gamma^\sigma},
  \qquad
  v\in V_\Xi.
\end{equation}
\end{theorem}

\begin{theorem}
\label{thm:inverse-sharpness}
The inverse inequality in \autoref{thm:inverse} is sharp as follows.
\begin{enumerate}
\item[\textup{(i)}] For every $\theta\in\mathbb S^{d-1}\times\mathbb R$ and $t\ge0$, $\phi_\theta\in\mathcal H_\gamma^t$ if and only if $t<k+1/2$.
\item[\textup{(ii)}] There is $\delta_0=\delta_0(k,C_b)>0$ such that, for $R\ge1$ and $0<\delta\le\delta_0$, one can choose a deterministic $\Xi_\delta\subset\Theta_R$ and $v_\delta\in V_{\Xi_\delta}$ with $\delta_\pm(\Xi_\delta)=\delta(\Xi_\delta)=\delta$. For all sufficiently large $R$, one can also choose a singleton $\Xi_R\subset\Theta_R$ and a $\mathrm{ReLU}^k$ neuron $v_R\in V_{\Xi_R}$. These functions satisfy
\begin{equation}
\label{eq:geometric-sharpness-all-orders}
  \frac{\|v_\delta\|_{\mathcal H_\gamma^\tau}}
       {\|v_\delta\|_{\mathcal H_\gamma^\sigma}}
  \asymp \delta^{-(\tau-\sigma)},
  \qquad
  \frac{\|v_R\|_{\mathcal H_\gamma^\tau}}
       {\|v_R\|_{\mathcal H_\gamma^\sigma}}
  \asymp (1+R)^{\tau-\sigma}.
\end{equation}
Both estimates hold simultaneously for all integers $0\le\sigma\le\tau\le k$. The $d_\pi$ distances between distinct parameters in $\Xi_\delta$ remain uniformly bounded away from zero.
\end{enumerate}
\end{theorem}

The proof in \autoref{subsec:inverse-parameters} compares dyadic Hermite Gram matrices using localized kernel estimates \cite{PetrushevXu2008}. Hermite spectral sums determine the regularity threshold, while finite differences and neurons with large biases establish sharpness of the two geometric powers.

\subsection{Weighted least squares convergence on \texorpdfstring{$\mathbb R^d$}{Rd}}

In this subsection, we use the inverse inequality to establish convergence of finite residual least squares for deterministic and random sampling, with consequences on fixed bounded regions.

Consider $\mathcal Au=f$ on $\mathbb R^d$, where $\mathcal A=\sum_{|\alpha|\le r}a_\alpha D^\alpha$ has order $r$.

Let $s_k=(d+2k+1)/2$. For $\Xi\subset\Theta_R$, define its fill distance relative to $\Theta_R$ by \[ h(\Xi):=\sup_{\theta\in\Theta_R}\min_{\xi\in\Xi}d_0(\theta,\xi). \] We use families $\{\Xi_{h,R}\}$ that are quasiuniform with respect to $d_0$, with uniform constants satisfying
\begin{equation}
\label{eq:ordinary-parameter-geometry}
  h(\Xi_{h,R})\asymp h,
  \qquad
  \delta(\Xi_{h,R})\asymp h,
  \qquad
  \#\Xi_{h,R}\asymp(1+R)h^{-d}.
\end{equation}

For deterministic residual samples $X=\{x_i\}_{i=1}^N\subset B_R$, define
\begin{equation} \label{eq:physical-sampling-geometry} h_X:=\sup_{x\in B_R}\min_i|x-x_i|, \qquad q_X:=\frac12\min_{i\ne j}|x_i-x_j|. \end{equation}
Assume $h_X\asymp q_X\asymp H$ and choose positive weights satisfying
\begin{equation} \label{eq:weight-condition} c_\omega H^d\rho_d(x_i) \le \omega_i \le C_\omega H^d\rho_d(x_i), \qquad \|g\|_{X,\omega}^2 := \sum_{i=1}^N\omega_i|g(x_i)|^2. \end{equation}

To control Gaussian tails outside $B_R$, set, for $M>0$,
\[ V_\Xi(M):=\left\{v=\sum_{\theta_j\in\Xi}c_j\phi_{\theta_j}:\sum_j|c_j|\le M\right\}. \]
The coefficient representation need not be unique; one admissible representation suffices. The weighted least squares approximation is
\begin{equation} \label{eq:constrained-LS-problem} u_h \in \operatorname*{arg\,min}_{v\in V_\Xi(M)} \|\mathcal Av-f\|_{X,\omega}^2. \end{equation}
A minimizer exists because the coefficient ball is compact and the residual objective is continuous.

We impose the following continuous stability and local growth assumptions.

\begin{assumption}
\label{ass:apriori}
Let $0\le r\le S<k$ be integers. Assume that the equation $\mathcal Au=f$ is uniquely solvable in the target solution class and that there exists a constant $C$, independent of $R$, the parameter set, and the network width, such that
\begin{equation*}
  \|w\|_{\mathcal H_\gamma^S}
  \le
  C
  \|\mathcal Aw\|_{\mathcal H_\gamma^{S-r}},
  \qquad
  w\in V_\Xi,
\end{equation*}
for every $R\ge1$ and every finite $\Xi\subset\Theta_R$.
\end{assumption}

\begin{assumption}
\label{ass:solution-local-growth}
Assume that there exist fixed exponents $p_1,p_2\ge0$ such that, for every $R\ge1$, $u|_{B_R}\in H^{s_k}(B_R)$ and
\begin{equation}
\label{eq:local-growth-u}
  \|u\|_{H^{s_k}(B_R)}
  \le
  C(1+R)^{p_1},
\end{equation}
and
\begin{equation*}
  \|D^\beta a_\alpha\|_{L^\infty(B_R)}
  \le
  C(1+R)^{p_2},
  \qquad
  |\alpha|\le r,\quad |\beta|\le k-r+1.
\end{equation*}
\end{assumption}

For $2\le S<k$, \autoref{ass:apriori} holds with $r=2$ for $\mathcal A=-\Delta+x\cdot\nabla+\lambda+b\cdot\nabla+c$ if $b,c$ and their derivatives through order $k-1$ are bounded and $\lambda>0$ is sufficiently large. This follows by a Neumann series perturbation of the shifted Ornstein--Uhlenbeck resolvent, and the coefficient bounds in \autoref{ass:solution-local-growth} hold with $p_2=1$. The unperturbed operator is also related to the Fokker--Planck equation $\partial_t p=\Delta p+\nabla\cdot(xp)$: writing $p=\rho_dq$ gives $\partial_t q=\Delta q-x\cdot\nabla q$. Thus an implicit Euler step for $q$ leads to a shifted Ornstein--Uhlenbeck equation; see \cite{DujardinHerauLafitte2020,FokGuoTang2002} for implicit discretizations of Fokker--Planck equations.

\begin{theorem}
\label{thm:convergence}
Let $0\le r\le S<k$ be integers with $k-r>(d-1)/2$, and suppose that \autoref{ass:apriori} and \autoref{ass:solution-local-growth} hold. Choose fixed $p_3,p_4>0$ and $\kappa>0$ such that
\[
 p_3(k-S+1/4)>p_2/2,\qquad
 p_4>(p_1+1/2)/2,\qquad
 \kappa^2>4(s_k-S),
\]
and set $R=R(h):=\kappa\sqrt{\log(e/h)}$ and $M_h=(\log(e/h))^{p_4}$.

Let $\Xi_{h,R}\subset\Theta_R$ satisfy \eqref{eq:ordinary-parameter-geometry}, with $\delta_\pm:=\delta_\pm(\Xi_{h,R})>0$. Choose $H\asymp\delta_\pm(\log(e/h))^{-p_3}$ and deterministic residual samples $X\subset B_R$ satisfying $h_X\asymp q_X\asymp H$, with $h_X,q_X$ defined in \eqref{eq:physical-sampling-geometry} and weights satisfying \eqref{eq:weight-condition}. For sufficiently small $h$, every minimizer in \eqref{eq:constrained-LS-problem}, with $\Xi=\Xi_{h,R}$ and $M=M_h$, satisfies
\begin{equation}
\label{eq:two-scale-convergence-rate}
  \|u-u_h\|_{\mathcal H_\gamma^S}
  \le
  C(\log(e/h))^{\chi}
  h^{s_k-r}\delta_\pm^{-(S-r)}.
\end{equation}
Here $\chi=p_3(S-r)+(p_1+p_2+s_k-r)/2$. The constant $C$ and the smallness threshold for $h$ depend only on $d,k,r,S,C_b$, $p_1,p_2,p_3,p_4,\kappa$, the constants in the two assumptions, and the fixed constants in the parameter geometry, sampling geometry, scale comparison for $H$, and weight bounds. In particular, $C$ is independent of $h,\delta_\pm,H,R,J,N$.

If, in addition, $\delta_\pm\gtrsim h$, then
\begin{equation}
\label{eq:width-rate}
  \|u-u_h\|_{\mathcal H_\gamma^S}
  \le
  C(1+\log J)^{\widetilde{\chi}}
  J^{-(s_k-S)/d},
\end{equation}
where $\widetilde{\chi}=\chi+(s_k-S)/(2d)$; the constant may additionally depend on the fixed lower bound for $\delta_\pm/h$. The shifted product dictionary in \autoref{lem:shifted-product-parameters} satisfies \eqref{eq:ordinary-parameter-geometry} and $\delta_\pm\asymp h$.
\end{theorem}

The following corollary gives convergence on fixed bounded regions for targets admitting a suitable extension to $\mathbb R^d$.

\begin{corollary}
\label{cor:fixed-region-restriction}
Let $\Omega\subset\mathbb R^d$ be a fixed bounded open set, with $H^S(\Omega)$ equipped with its usual derivative norm, and let $u^\Omega$ be a target on $\Omega$. Assume that there exist a solution $u$ on $\mathbb R^d$ and associated data and discretizations satisfying all hypotheses of \autoref{thm:convergence}, with $u|_\Omega=u^\Omega$. For any minimizer $u_h$ in \eqref{eq:constrained-LS-problem}, set $u_h^\Omega:=u_h|_\Omega$. Then
\begin{equation}
\label{eq:fixed-region-two-scale-rate}
  \|u^\Omega-u_h^\Omega\|_{H^S(\Omega)}
  \le C_\Omega(\log(e/h))^\chi
  h^{s_k-r}\delta_\pm^{-(S-r)}.
\end{equation}
If, in addition, $\delta_\pm\gtrsim h$, then
\begin{equation}
\label{eq:fixed-region-one-scale-rates}
  \|u^\Omega-u_h^\Omega\|_{H^S(\Omega)}
  \le C_\Omega(1+\log J)^{\widetilde{\chi}}J^{-(s_k-S)/d}.
\end{equation}
The constant $C_\Omega$ is independent of $h,\delta_\pm,H,R,J,N$.
\end{corollary}

The following theorem extends the convergence estimates to independent uniform residual samples.

\begin{theorem}
\label{thm:random-physical-least-squares}
Retain the assumptions on the orders, equation, solution, and parameter dictionary, and the choices of $p_3,p_4,R,M_h$, from \autoref{thm:convergence}. Let $Y_1,\ldots,Y_N$ be independent and uniformly distributed in $B_R$, and define $\|g\|_{N,\gamma}^2:=(|B_R|/N)\sum_{i=1}^N\rho_d(Y_i)|g(Y_i)|^2$. Let $u_{h,N}$ minimize $\|\mathcal Av-f\|_{N,\gamma}^2$ over $v\in V_{\Xi_{h,R}}(M_h)$, with $f$ taken as the continuous representative of $\mathcal Au$.

There exist constants $C_{\rm samp}\ge1$ and $C_{\rm err}>0$ such that, if $\eta\in(0,1)$ and
\[
 N\ge C_{\rm samp}\delta_\pm^{-d}(\log(e/h))^{d(p_3+1/2)}
 \log\bigl(C_{\rm samp}/(\eta\delta_\pm)\bigr),
\]
then, for sufficiently small $h$, with probability at least $1-\eta$ every minimizer satisfies
\begin{equation}
\label{eq:random-two-scale-rate}
  \|u-u_{h,N}\|_{\mathcal H_\gamma^S}
  \le
  C_{\rm err}(\log(e/h))^\chi
  h^{s_k-r}\delta_\pm^{-(S-r)}.
\end{equation}
The constants $C_{\rm samp},C_{\rm err}$ and the smallness threshold for $h$ depend only on $d,k,r,S,C_b$, $p_1,p_2,p_3,p_4,\kappa$, the constants in the two assumptions, and the fixed parameter geometry constants. They are independent of $h,\delta_\pm,R,J,N,\eta$.

If, in addition, $\delta_\pm\gtrsim h$, the same event yields
\begin{equation}
\label{eq:random-width-rate}
  \|u-u_{h,N}\|_{\mathcal H_\gamma^S}
  \le
  C_{\rm err}(1+\log J)^{\widetilde{\chi}}J^{-(s_k-S)/d}.
\end{equation}
Here $C_{\rm err}$ may additionally depend on the fixed lower bound for $\delta_\pm/h$. In this case, for fixed $\eta$, one may choose $N=O(J(\log J)^{dp_3+(d+1)/2})$.
\end{theorem}

Classical estimates for Kolmogorov widths of Sobolev classes \cite{Pinkus1985,EdmundsTriebel1996} yield the following lower bound, which matches the convergence exponent in \eqref{eq:width-rate} and \eqref{eq:random-width-rate}.

\begin{proposition}
\label{prop:width-lower-bound}
There exists a class $\mathcal U_{s_k}$ of functions supported in $B_{1/2}$ with $\|u\|_{H^{s_k}(\mathbb R^d)}\le1$ such that, for every integer $0\le S<k$ and all $J\ge1$,
\begin{equation}
\label{eq:width-lower-bound}
  \inf_{\substack{W\subset\mathcal H_\gamma^S\\ \dim W\le J}}
  \ \sup_{u\in\mathcal U_{s_k}}
  \ \inf_{w\in W}
  \|u-w\|_{\mathcal H_\gamma^S}
  \ge
  cJ^{-(s_k-S)/d},
\end{equation}
where $c>0$ depends only on $d,k,S$. If a fixed operator $\mathcal A$ and orders $r,S$ satisfy the hypotheses of \autoref{thm:convergence} and \autoref{thm:random-physical-least-squares} uniformly for $u\in\mathcal U_{s_k}$, dictionaries with $\delta_\pm\asymp h$ yield matching upper bounds up to logarithmic factors. Their width exponent is therefore optimal among fixed linear trial spaces. This applies in particular to $\mathcal A=I$ when $k>(d-1)/2$.
\end{proposition}

\section{Numerical Experiments}
\label{sec:experiments}

In this section, we relate convergence of the least squares approximations to the geometry of the network parameters and the density of the residual samples. We first examine convergence for a model PDE, then use parameter comparisons and collisions to study the role of separation. We also examine discrete stability under deterministic and random residual sampling. The code is available at \url{https://github.com/hanfei27/NNLS}.

\subsection{Model problem and discretization}
\label{subsec:experiment-setup}

In this subsection, we specify a model problem for which the PDE stability estimate is exact. We consider the shifted Ornstein--Uhlenbeck equation
\[
  \cA u=-\Delta u+x\cdot\nabla u+u=f\quad\text{in }\mathbb R^2.
\]
Since $\cA H_\alpha=(|\alpha|+1)H_\alpha$, we have $\|\cA w\|_{\cH_\gamma^{S-2}}=\|w\|_{\cH_\gamma^S}$ for $S\ge2$. We choose $f=\cA u$ and measure the relative errors
\[
  E_s(u_h):=
  \frac{\|u-u_h\|_{H_\gamma^s}}
       {\|u\|_{H_\gamma^s}},
  \qquad s=0,1,2,3.
\]
By \eqref{eq:Hermite-derivative-norm-equivalence}, the derivative norm is equivalent to the $\mathcal H_\gamma^s$ norm. The five exact solutions are
\begin{align*}
 u_1(x,y)&=(1+0.03x+0.025y+0.015xy)e^{-0.10x^2-0.12y^2},\\
 u_2(x,y)&=\exp\!\left[-0.50\!\left(\sqrt{1+0.16(x+0.25y-0.10)^2+0.20(-0.20x+y+0.10)^2}-1\right)\right],\\
 u_3(x,y)&=0.70\tanh(0.14x-0.08y)+0.30\tanh(0.08x+0.15y),\\
 u_4(x,y)&=\operatorname{arsinh}(0.16x+0.10y),\\
 u_5(x,y)&=|x|^{21/4}e^{-0.10x^2-0.12y^2}.
\end{align*}

We use $\mathrm{ReLU}^4$ networks $u_h[c]=\sum_{j=1}^Jc_j\phi_j$, with $\phi_j(x)=(a_j\cdot x-b_j)_+^4$. Unless stated otherwise, residuals are sampled in $B_{5.5}$ and the parameters lie in the fixed window $\Theta_0=\mathbb S^1\times[-5,5]$. Following the construction in \autoref{lem:shifted-product-parameters}, we use quasiuniform parameter sets on $\Theta_0$ of the form
\[
 a_p=\left(\cos\frac{2\pi p}{n_a},\sin\frac{2\pi p}{n_a}\right),
 \qquad b_t=-5+\left(t+\frac34\right)\frac{10}{n_b}.
\]
Here $p=0,\ldots,n_a-1$ and $t=0,\ldots,n_b-1$; the bias shift avoids exact antipodal pairings. We denote the fill distance on $\Theta_0$, measured with $d_0$, by $h_{\Theta_0}$. Deterministic residual samples use tensor Gauss--Legendre rules with nine nodes on rotated Cartesian cells, retaining points in $B_{5.5}$ and multiplying their weights by $\rho_d$.

Unconstrained least squares problems are solved by QR factorization with column pivoting \cite{BusingerGolub1965}. For coefficient-constrained problems, when the unconstrained solution violates the bound, we use accelerated projected gradient iterations on the $\ell^1$ ball \cite{BeckTeboulle2009}, with termination assessed by the Karush--Kuhn--Tucker conditions. Relative errors are evaluated by composite Gauss--Legendre quadrature on $[-8,8]^2$. For the study of parameter refinement, truncation sensitivity is checked by enlarging this domain with the computed coefficients fixed.

We compute numerical ranks and condition numbers from the residual matrix $A_{ij}=\omega_i^{1/2}(\mathcal A\phi_{a_j,b_j})(x_i)$ after column normalization, $\widetilde A_{\cdot j}=A_{\cdot j}/\|A_{\cdot j}\|_2$. The numerical rank counts singular values of $\widetilde A$ exceeding $\max\{N,J\}\varepsilon$ times the largest singular value, where $\varepsilon$ denotes machine precision. The plotted condition number $\kappa_1$ is the ratio of the largest to the smallest singular value, capped at $1/(\max\{N,J\}\varepsilon)$; matrices without full column rank are assigned this cap.

\subsection{Convergence under parameter refinement}
\label{subsec:experiment-convergence}

In this subsection, we study convergence as the parameter set is refined. For $d=2$, $k=4$, and $r=2$, the estimate in \autoref{thm:convergence} gives the factors $h^{7/2}$ and $h^{5/2}$ at orders $S=2$ and $S=3$ when coverage and separation have comparable scales. We examine the corresponding behavior on the fixed parameter window $\Theta_0$.

\begin{figure}[!t]
\centering
\includegraphics[width=\linewidth]{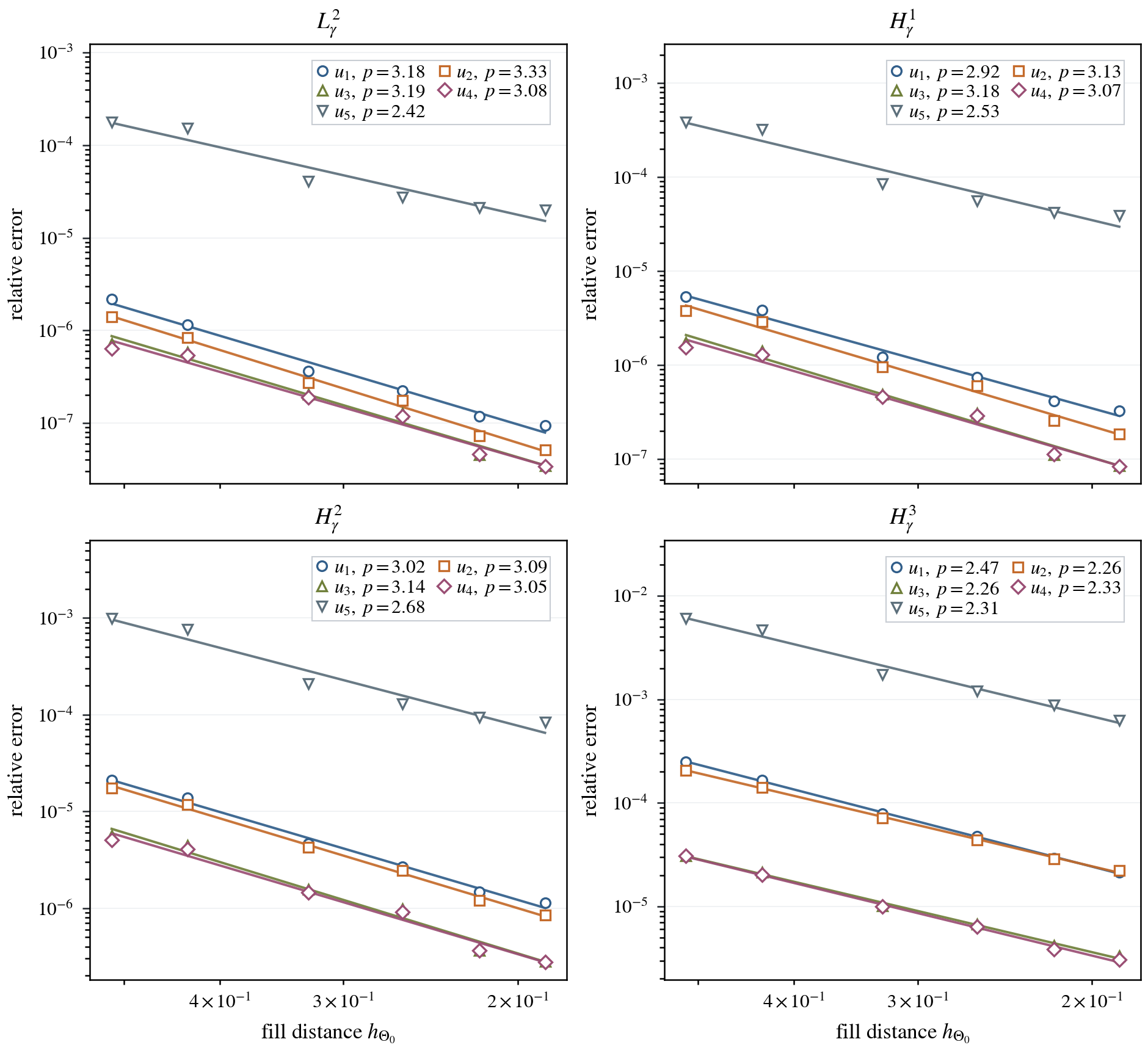}
\caption{Relative errors $E_s$, $s=0,1,2,3$, under parameter refinement on $\Theta_0$; $p$ is the least squares slope of $\log E_s$ against $\log h_{\Theta_0}$.}
\label{fig:experiment-sobolev-convergence}
\end{figure}

The six network widths are $J=420,600,1056,1640,2352,3192$, with corresponding residual sample sizes $N=1680,2380,4128,6392,9120,12320$. \autoref{fig:experiment-sobolev-convergence} shows systematic error decay under parameter refinement for all five solutions in all four norms. For $u_1,\ldots,u_4$, the fitted orders are about $3.0$--$3.1$ for $E_2$ and $2.3$--$2.5$ for $E_3$, below the reference exponents $7/2$ and $5/2$ over this short range of $h_{\Theta_0}$ with $R=5.5$ and the bias window fixed; the order drops by $0.55$--$0.88$ from $E_2$ to $E_3$, compared with the predicted drop of one. For these four solutions, the unconstrained least squares coefficients in this refinement study satisfy $\|c\|_1\le0.19$; since $M_h\ge1$, the constraint is inactive and the computed solution is also a minimizer of \eqref{eq:constrained-LS-problem}.

\subsection{Choice of network parameters}
\label{subsec:experiment-parameter-design}

In this subsection, we examine the effect of parameter geometry at fixed width $J=600$ and with the same $N=2380$ residual samples. We compare the quasiuniform grid of $24$ directions and $25$ biases with two random constructions: $24$ independent uniform directions combined with the same biases, and $600$ independent uniform parameters in $[0,2\pi)\times[-5,5]$. 

\begin{table}[H]
\centering
\caption{Relative errors averaged over the four smooth targets $u_1,\ldots,u_4$. Random rows report mean $\pm$ standard deviation over $50$ parameter sets.}
\label{tab:experiment-parameter-design-summary}
\scriptsize
\renewcommand{\arraystretch}{1.12}
\setlength{\tabcolsep}{2pt}
\begin{tabular}{@{}lcccc@{}}
\toprule
Parameter design & $E_0$ & $E_1$ & $E_2$ & $E_3$\\
\midrule
Quasiuniform
  &$7.774\times10^{-7}$
  &$2.361\times10^{-6}$
  &$8.534\times10^{-6}$
  &$8.700\times10^{-5}$\\
Random directions
  &$(1.869\pm1.735)\times10^{-6}$
  &$(6.383\pm6.159)\times10^{-6}$
  &$(2.866\pm2.215)\times10^{-5}$
  &$(3.067\pm1.927)\times10^{-4}$\\
Fully random
  &$(2.494\pm0.648)\times10^{-6}$
  &$(8.586\pm1.964)\times10^{-6}$
  &$(4.402\pm0.819)\times10^{-5}$
  &$(4.342\pm0.715)\times10^{-4}$\\
\bottomrule
\end{tabular}
\end{table}

\autoref{tab:experiment-parameter-design-summary} shows that the quasiuniform construction achieves the smallest mean errors in all four norms. Since both coverage and separation improve relative to the random designs, we next isolate the effect of separation on the inverse constant.

\subsection{Growth of the inverse constant}
\label{subsec:experiment-d0-sharpness}

In this subsection, we test the separation exponent in \autoref{thm:inverse-sharpness} using six neurons $\phi_j(x)=(x_1-(j-5/2)\delta)_+^4$, $j=0,\ldots,5$. For each $\tau=1,2,3,4$, the inverse constant is the maximum of $\frac{\norm{v}_{\cH_\gamma^\tau}}{\norm{v}_{L_\gamma^2}}$ over nonzero functions in their span. Its predicted growth is $\delta^{-\tau}$.

We use $\delta=2^{-3},\ldots,2^{-7}$ and evaluate the spectral norms in \eqref{eq:Hermite-Sobolev-definition} from derivative Gram matrices in higher precision. \autoref{tab:experiment-inverse-orders} supports the sharpness of the separation exponent. The inverse constant grows at fixed width, showing that width alone cannot control inverse stability.

\begin{table}[H]
\centering
\caption{Growth exponents in $1/\delta$.}
\label{tab:experiment-inverse-orders}
\begin{tabular}{@{}cc@{}}
\toprule
Predicted exponent $\tau$ & Fitted exponent\\
\midrule
$1$ & $0.9232$\\
$2$ & $1.9106$\\
$3$ & $2.9069$\\
$4$ & $3.9056$\\
\bottomrule
\end{tabular}
\end{table}

\subsection{Residual sampling}
\label{subsec:experiment-oversampling}

We examine how residual sampling affects stability and accuracy for a fixed trial space. We use the $16\times17$ parameter grid with $J=272$ and vary the number of residual samples as listed in \autoref{tab:experiment-sampling-ranks}. For $u=u_1$ and $R=5.5$, we compare deterministic sampling with independent uniform sampling. The random approximations use the least squares formulation in \autoref{thm:random-physical-least-squares} with the coefficient bound $M=1$ and $30$ independent realizations at each sample size. We report the numerical rank and condition number of the column-normalized residual matrix, together with the relative error $E_3$.

\begin{figure}[H]
\centering
\includegraphics[width=\linewidth]{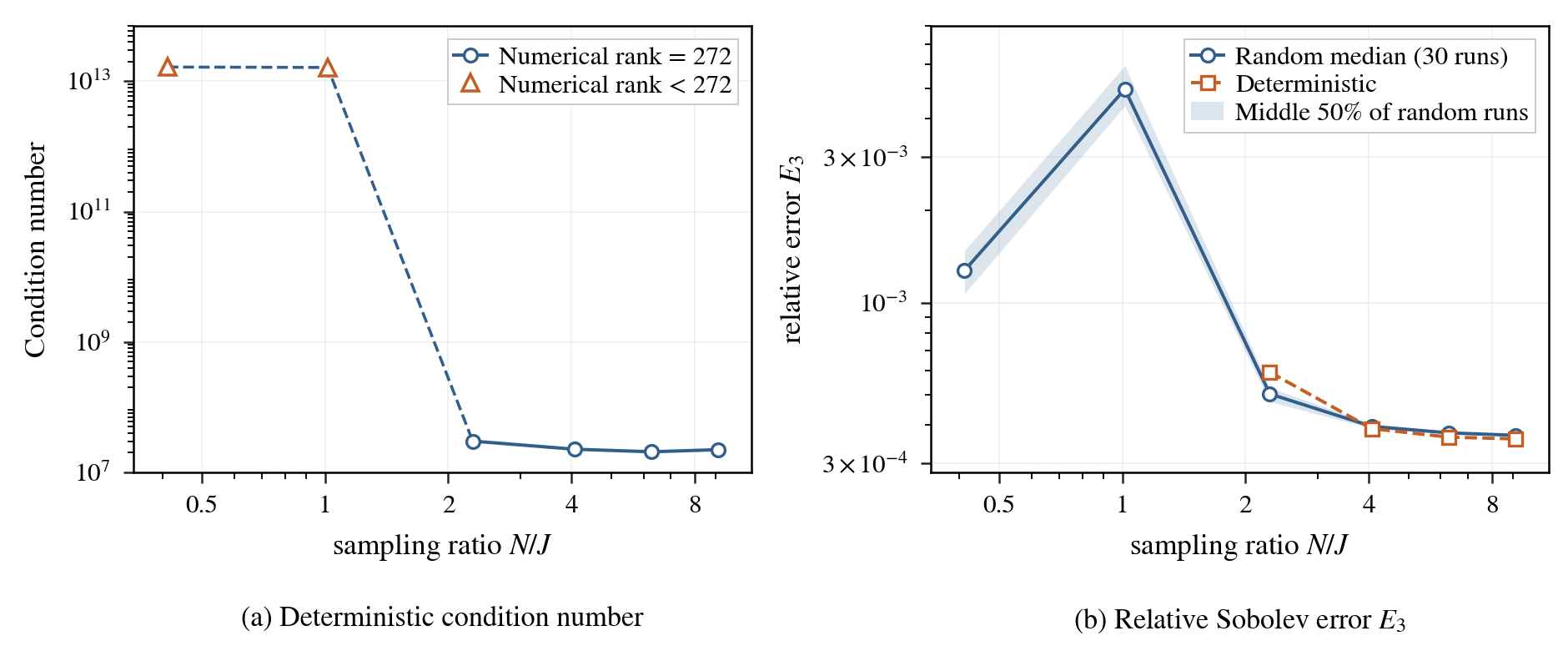}
\caption{Effect of residual sampling on conditioning and solution accuracy. Triangles mark numerical rank deficiency; their vertical positions indicate the display cap.}
\label{fig:experiment-residual-sampling}
\end{figure}

\begin{table}[H]
\centering
\caption{Deterministic ranks and full rank counts for random sampling, with $J=272$.}
\label{tab:experiment-sampling-ranks}
\begin{tabular}{@{}lrrrrrr@{}}
\toprule
$N$ & $112$ & $276$ & $624$ & $1108$ & $1708$ & $2480$\\
\midrule
Deterministic rank & $112$ & $266$ & $272$ & $272$ & $272$ & $272$\\
Full rank (random) & $0/30$ & $29/30$ & $30/30$ & $30/30$ & $30/30$ & $30/30$\\
\bottomrule
\end{tabular}
\end{table}

\autoref{tab:experiment-sampling-ranks} and \autoref{fig:experiment-residual-sampling} show that a sample count slightly larger than the network width can still yield poor stability and accuracy. At $N=276$, the deterministic matrix remains numerically rank deficient, while the median random error peaks despite full numerical rank in $29$ of the $30$ runs. Increasing $N$ to $624$ restores full numerical rank in all tested cases and substantially reduces the random error. With further sampling, the random and deterministic errors become comparable and improve only modestly.

\subsection{Antipodal parameter collisions}
\label{subsec:experiment-antipodal}

In this subsection, we keep the residual set dense and examine the effect of nearly antipodal parameters. We fix $N=4112$ and vary the antipodal separation of a $J=132$ parameter set. We use the matrix normalization and numerical rank criterion defined in \autoref{subsec:experiment-setup}.

\autoref{fig:experiment-antipodal-stability} shows that nearly antipodal parameters cause a substantial deterioration in conditioning before numerical rank is lost, even with dense residual sampling. For the tested separations from $9.09\times10^{-3}$ down to $9.09\times10^{-6}$, each tenfold reduction in $\delta_\pm$ increases $\kappa_1$ by approximately a factor of $10$, consistent with the empirical scaling $\kappa_1\propto\delta_\pm^{-1}$ over this range.

\begin{figure}[H]
\centering
\includegraphics[width=0.6\linewidth]{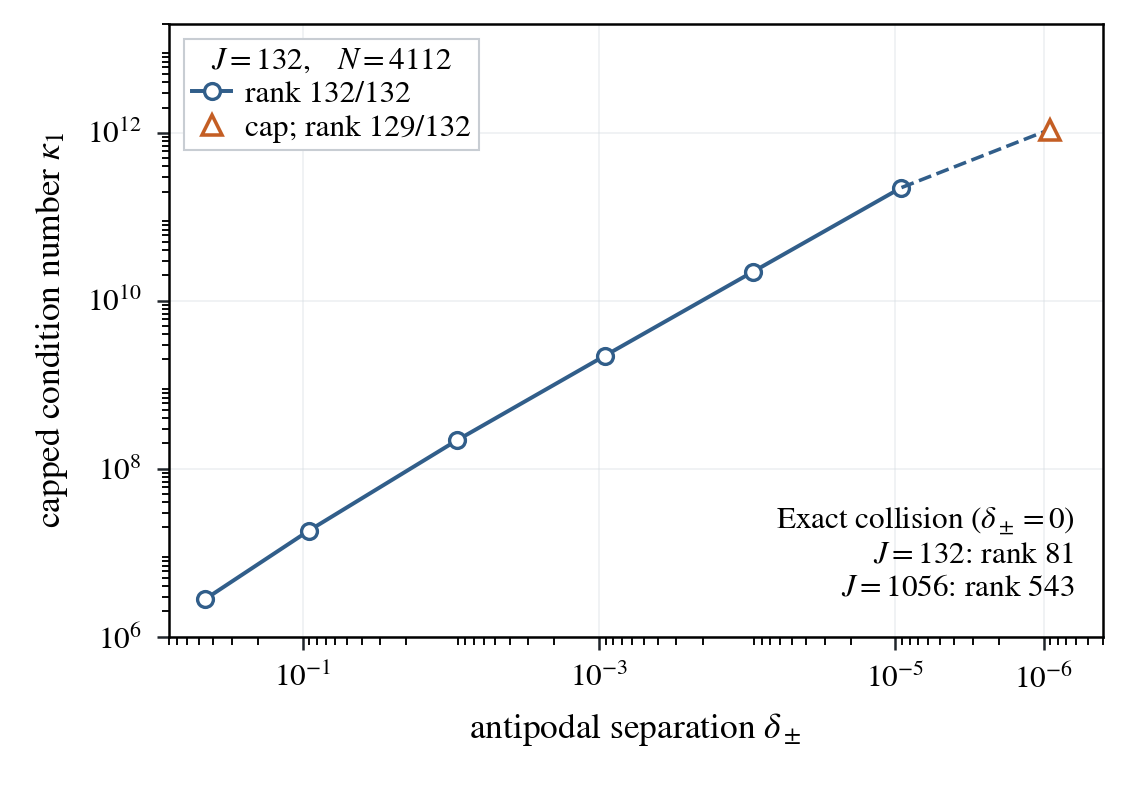}
\caption{Capped condition number near antipodal parameter collisions.}
\label{fig:experiment-antipodal-stability}
\end{figure}

At exact antipodal pairing, the identity $\phi_{a,b}+\phi_{-a,-b}=(a\cdot x-b)^4$ gives $\dim V_\Xi\le P+15$ for $J=2P$ neurons in two dimensions. The computed ranks $81$ and $543$, for $J=132$ and $J=1056$, respectively, attain this bound. This algebraic dependence persists regardless of the number of residual samples.

\section{Proofs of the Main Results}
\label{sec:proofs}

In this section, we prove the main theorems. The technical estimates supporting the proofs are collected in \autoref{app:inverse-lemmas} and \autoref{app:convergence-lemmas}.

\subsection{Inverse inequality and sharpness}
\label{subsec:inverse-parameters}

In this subsection, we compare Hermite Gram matrices to prove the inverse inequality and then establish sharpness.

For $R\ge1$ and $\Xi=\{\theta_j=(a_j,b_j)\}_{j=1}^J\subset\Theta_R$, set $\eta_n(b):=\int_{\mathbb R}(z-b)_+^k H_n(z)\,d\gamma_1(z)$. Let $v=\sum_{j=1}^Jc_j\phi_{a_j,b_j}$ and $c=(c_1,\ldots,c_J)^\top$. Since $(z-b)_+^k\in L^2(\gamma_1)$, the definition of $\eta_n(b)$ and the orthonormal basis property of $\{H_n\}_{n\ge0}$ give
\[
  (z-b)_+^k=\sum_{n\ge0}\eta_n(b)H_n(z)\qquad\text{in }L^2(\gamma_1).
\]
Since $|a_j|=1$, the projection $a_j\cdot x$ has distribution $\gamma_1$ under $\gamma$. Thus, for every $F\in L^2(\gamma_1)$,
\[
  \int_{\mathbb R^d}|F(a_j\cdot x)|^2\,d\gamma(x)=\int_{\mathbb R}|F(z)|^2\,d\gamma_1(z).
\]
Applying this identity to the remainders of the Hermite expansion in one variable shows that substituting $z=a_j\cdot x$ preserves $L^2$ convergence. Combining this expansion with the Hermite addition formula \cite[Eq.~(41)]{UmemuraKono1965}, iterated over the coordinates and expressed in our normalization, we obtain, in $L^2(\gamma)$,
\[
  \phi_{a_j,b_j}(x)=\sum_{n\ge0}\eta_n(b_j)H_n(a_j\cdot x)=\sum_{n\ge0}\eta_n(b_j)\sum_{|\alpha|=n}\sqrt{\frac{n!}{\alpha!}}\,a_j^\alpha H_\alpha(x).
\]
Consequently, for $|\alpha|=n$,
\[
  \widehat v_\alpha=\sqrt{\frac{n!}{\alpha!}}\sum_{j=1}^Jc_j\eta_n(b_j)a_j^\alpha.
\]
The multinomial formula then yields
\[
  \sum_{|\alpha|=n}|\widehat v_\alpha|^2=\sum_{i,j=1}^Jc_ic_j\eta_n(b_i)\eta_n(b_j)\sum_{|\alpha|=n}\frac{n!}{\alpha!}a_i^\alpha a_j^\alpha=\sum_{i,j=1}^Jc_ic_j\eta_n(b_i)\eta_n(b_j)(a_i\cdot a_j)^n.
\]
Define the Gram matrix at degree $n$ by
\begin{equation}
\label{eq:An-definition}
  (A_n)_{ij}:=\eta_n(b_i)\eta_n(b_j)(a_i\cdot a_j)^n.
\end{equation}
Then $c^\top A_nc=\sum_{|\alpha|=n}|\widehat v_\alpha|^2\ge0$, so $A_n$ is positive semidefinite. For $0\le s<k+1/2$, \autoref{lem:normalized-block-localization} gives $\phi_{\theta_j}\in\mathcal H_\gamma^s$, and the Cauchy--Schwarz inequality implies absolute entrywise convergence of $K_s:=\sum_{n\ge0}(1+n)^sA_n$. Using \eqref{eq:Hermite-Sobolev-definition}, we obtain
\begin{equation}
\label{eq:quadratic-norm-identities}
  \|v\|_{\mathcal H_\gamma^s}^2=\sum_{n\ge0}(1+n)^s\sum_{|\alpha|=n}|\widehat v_\alpha|^2=\sum_{n\ge0}(1+n)^s c^\top A_nc=c^\top K_sc,\qquad \|v\|_{L_\gamma^2}^2=c^\top K_0c.
\end{equation}
Define $\widetilde\phi_{a,b}:=\rho_1(b)^{-1/2}\phi_{a,b}$ and $D:=\operatorname{diag}(\rho_1(b_1)^{-1/2},\ldots,\rho_1(b_J)^{-1/2})$. The corresponding matrices are
\begin{equation}
\label{eq:normalized-matrix-definitions}
  \widetilde A_n:=DA_nD,\qquad
  \widetilde K_s:=DK_sD.
\end{equation}
Since $D$ is invertible, a comparison of $\widetilde K_\tau$ and $\widetilde K_\sigma$ transfers to $K_\tau$ and $K_\sigma$ and hence to the Sobolev norms in \eqref{eq:quadratic-norm-identities}. We use the positive semidefinite order: for symmetric matrices $P,Q$, the relation $P\preceq Q$ means $c^\top Pc\le c^\top Qc$ for every $c\in\mathbb R^J$.

Choose a nonincreasing $\vartheta\in C_c^\infty([0,\infty))$ with $0\le\vartheta\le1$, $\vartheta=1$ on $[0,1]$, and $\vartheta=0$ on $[2,\infty)$. Set $\zeta_0(t):=\vartheta(2t)$ and $\zeta(t):=\vartheta(t)-\vartheta(2t)$. These nonnegative windows satisfy $\operatorname{supp}\zeta_0\subset[0,1]$, $\operatorname{supp}\zeta\subset[1/2,2]$, and
\begin{equation}
\label{eq:dyadic-partition}
  \zeta_0(t)+\sum_{j=0}^{\infty}\zeta(2^{-j}t)=1,\qquad t\ge0.
\end{equation}
For dyadic $N=2^j$, $j\ge0$, define
\[
  M_N:=\sum_{n\ge0}\zeta(n/N)A_n,\qquad
  K_{s,N}:=\sum_{n\ge0}(1+n)^s\zeta(n/N)A_n.
\]
Since $\zeta_0(0)=1$ and $\zeta_0(n)=0$ for integers $n\ge1$, the zeroth block is
\[
  K_{s,0}:=\sum_{n\ge0}(1+n)^s\zeta_0(n)A_n=A_0.
\]
Tildes denote congruence by $D$. Thus $\widetilde K_s=\widetilde K_{s,0}+\sum_{N=2^j}\widetilde K_{s,N}$. With $\widetilde\eta_n(b):=\rho_1(b)^{-1/2}\eta_n(b)$, define
\[
  \widetilde m_N(\theta,\theta')
  :=\sum_{n\ge0}\zeta(n/N)\widetilde\eta_n(b)\widetilde\eta_n(b')(a\cdot a')^n,
\]
so that $(\widetilde M_N)_{ij}=\widetilde m_N(\theta_i,\theta_j)$. The following estimate combines the Hermite coefficient formula with localized Hermite kernel bounds \cite{PetrushevXu2008}; the antipodal term follows from the parity of the Hermite polynomials. Its proof is given in \autoref{app:inverse-lemmas}.

\begin{lemma}
\label{lem:normalized-block-localization}
There exists a fixed $N_1\ge1$ such that, for every integer $L\ge1$, there exists $C_L>0$ satisfying, for all $R\ge1$, $N\ge N_1$, and $\theta,\theta'\in\Theta_R$,
\begin{equation}
\label{eq:normalized-block-localization}
  |\widetilde m_N(\theta,\theta')|
  \le
  C_L N^{-k-1/2}
  (1+\sqrt N\,d_0(\theta,\theta'))^{-L}
  +
  C_L N^{-k-1/2}
  (1+\sqrt N\,d_\pi(\theta,\theta'))^{-L}.
\end{equation}
Moreover, there is a fixed $C>0$ such that $N\ge C(1+R)^2$ implies
\[
  \widetilde m_N(\theta,\theta)
  \ge
  cN^{-k-1/2},
  \qquad \theta\in\Theta_R .
\]
In particular, $\phi_\theta\in\mathcal H_\gamma^s$ for every $\theta\in\Theta_R$ and $0\le s<k+1/2$.
\end{lemma}

The next lemma combines these kernel estimates with parameter separation to bound the Hermite Gram matrices from above and below. Set
\begin{equation}
\label{eq:truncated-antipodal-separation}
  \bar\delta_\pm:=\min\{1,\delta_\pm(\Xi)\}.
\end{equation}

\begin{lemma}
\label{lem:center-dominance}
Assume $\delta_\pm(\Xi)>0$. There exist $c,C>0$ and $\Lambda\ge1$, depending only on $d,k,C_b$, such that, for every dyadic $N\ge\Lambda\max\{N_1,\bar\delta_\pm^{-2},(1+R)^2\}$,
\begin{equation}
\label{eq:center-dominance}
  cN^{-k-1/2}I\preceq\widetilde M_N\preceq CN^{-k-1/2}I.
\end{equation}
\end{lemma}

\begin{proof}
Define
\begin{equation}
\label{eq:Nstar-definition}
  N_*
  :=
  2^{\left\lceil
  \log_2\left(
  \Lambda\max\{N_1,\bar\delta_{\pm}^{-2},(1+R)^2\}
  \right)
  \right\rceil}.
\end{equation}
Then $N_*$ is dyadic and
\begin{equation*}
  \Lambda\max\{N_1,\bar\delta_{\pm}^{-2},(1+R)^2\}
  \le
  N_*
  <
  2\Lambda\max\{N_1,\bar\delta_{\pm}^{-2},(1+R)^2\}.
\end{equation*}
Fix a dyadic $N\ge N_*$. For sufficiently large $\Lambda$, \autoref{lem:normalized-block-localization} gives, with fixed $c>0$,
\[
  cN^{-k-1/2}\le(\widetilde M_N)_{ii}\le CN^{-k-1/2}.
\]
Choose an integer $L>d$. The localization estimate \eqref{eq:normalized-block-localization} and \autoref{lem:packing} imply
\[
  \sum_{j\ne i}|(\widetilde M_N)_{ij}|\le C_LN^{-k-1/2}(\sqrt N\,\bar\delta_\pm)^{-L}\le C_L\Lambda^{-L/2}N^{-k-1/2}.
\]
Choose $\Lambda$ so large that $C_L\Lambda^{-L/2}\le c/2$. Gershgorin's theorem \cite[Section~6.1]{HornJohnson2013} then gives both bounds in \eqref{eq:center-dominance}.
\end{proof}

\begin{proof}[Proof of \autoref{thm:inverse}]
Fix $0\le\sigma\le\tau<k+1/2$ and let $N_*$ be the dyadic frequency in \eqref{eq:Nstar-definition}. By \autoref{lem:center-dominance},
\[
  \widetilde M_{N_*}\succeq cN_*^{-k-1/2}I,
  \qquad
  N_*^{1/2}\asymp\max\{\bar\delta_\pm^{-1},1+R\}
  \asymp\max\{\delta_\pm^{-1},1+R\}.
\]
Since $1+n\asymp N_*$ on the support of $\zeta(n/N_*)$,
\begin{equation}
\label{eq:compact-Ksigma-lower}
  \widetilde K_\sigma
  \succeq cN_*^\sigma\widetilde M_{N_*}
  \succeq cN_*^{\sigma-k-1/2}I.
\end{equation}

We estimate the blocks with $N\le N_*$ and $N>N_*$ separately. In the first case, every contributing $n$ satisfies $n\le2N_*$, hence
\[
  \widetilde K_{\tau,0}+\sum_{N\le N_*}\widetilde K_{\tau,N}
  \preceq C N_*^{\tau-\sigma}\widetilde K_\sigma.
\]
For a dyadic block $N>N_*$, \autoref{lem:center-dominance} gives $\|\widetilde M_N\|_{\ell^2\to\ell^2}\le CN^{-k-1/2}$. Since each $\widetilde A_n\succeq0$, $\zeta\ge0$, and $(1+n)^\tau\asymp N^\tau$ on this block,
\[
 0\preceq\widetilde K_{\tau,N}
 \preceq C N^\tau\widetilde M_N.
\]
Consequently,
\[
  \|\widetilde K_{\tau,N}\|_{\ell^2\to\ell^2}
  \le CN^{\tau-k-1/2}.
\]
The exponent is negative, so dyadic summation yields
\[
  0\preceq\sum_{N>N_*}\widetilde K_{\tau,N}
  \preceq C N_*^{\tau-k-1/2}I
  \preceq C N_*^{\tau-\sigma}\widetilde K_\sigma,
\]
where the last step uses \eqref{eq:compact-Ksigma-lower}. Consequently,
\[
  \widetilde K_\tau
  \preceq C N_*^{\tau-\sigma}\widetilde K_\sigma.
\]
By \eqref{eq:normalized-matrix-definitions}, congruence by $D^{-1}$ gives the same inequality for $K_\tau$ and $K_\sigma$. Taking quadratic forms and square roots,
\[
  \|v\|_{\mathcal H_\gamma^\tau}
  \le C N_*^{(\tau-\sigma)/2}
  \|v\|_{\mathcal H_\gamma^\sigma}
  \le C\max\{\delta_\pm^{-1},1+R\}^{\tau-\sigma}
  \|v\|_{\mathcal H_\gamma^\sigma}.
\]
This proves \eqref{eq:scale-inverse}.
\end{proof}

\begin{proof}[Proof of \autoref{thm:inverse-sharpness}]
Fix $\theta=(a,b)$. Define the Hermite functions normalized in $L^2(\mathbb R,dx)$ by
\begin{equation}
\label{eq:psi-definition}
  \psi_m(x):=\rho_1(x)^{1/2}H_m(x).
\end{equation}
By \autoref{lem:eta-formula}, for $m\ge0$,
\begin{equation}
\label{eq:endpoint-eta-block}
  |\eta_{m+k+1}(b)|^2
  =(k!)^2\rho_1(b)
  \frac{m!}{(m+k+1)!}\,\psi_m(b)^2.
\end{equation}
Take $N_j=3^jN_0$, $j\ge0$, with $N_0$ sufficiently large, depending on $b$, for \autoref{lem:Hermite-band-lower} to apply with $\alpha=1$ and $\beta=2$. Whenever $N_j\le m\le2N_j$,
\[
  \frac{m!}{(m+k+1)!}\asymp N_j^{-k-1},
  \qquad
  (m+k+2)^{k+1/2}\asymp N_j^{k+1/2}.
\]
Combining \eqref{eq:endpoint-eta-block} with \eqref{eq:Hermite-band-lower}, every one of these pairwise disjoint frequency bands satisfies
\[
  \sum_{N_j\le m\le2N_j}
  (m+k+2)^{k+1/2}|\eta_{m+k+1}(b)|^2
  \ge
  c\rho_1(b)N_j^{-1/2}
  \sum_{N_j\le m\le2N_j}\psi_m(b)^2
  \ge 
  c\rho_1(b)>0.
\]
Summing over $j$ gives $\sum_{n=0}^\infty(1+n)^{k+1/2}|\eta_n(b)|^2=+\infty,$ so $\phi_\theta\notin\mathcal H_\gamma^{k+1/2}$. By \autoref{lem:normalized-block-localization}, $\phi_\theta\in\mathcal H_\gamma^t$ for every $t<k+1/2$. The nesting of the spectral spaces then excludes every $t\ge k+1/2$, proving part \textup{(i)}.

For the first estimate in part \textup{(ii)}, fix $a\in\mathbb S^{d-1}$ and take $\delta_0=1/(k+1)$. For $0<\delta\le\delta_0$, set $b_\ell=\ell\delta$, $0\le\ell\le k+1$. These biases lie in $[0,1]$ and hence define parameters in $\Theta_R$. Define
\[
  \Xi_\delta:=\{(a,b_\ell):0\le\ell\le k+1\},
  \qquad
  v_\delta(x):=\sum_{\ell=0}^{k+1}
  (-1)^\ell\binom{k+1}{\ell}
  (a\cdot x-\ell\delta)_+^k.
\]
For $i\ne j$,
\[
  d_0((a,b_i),(a,b_j))=|i-j|\delta,
  \qquad
  d_\pi((a,b_i),(a,b_j))=\pi+(i+j)\delta\ge\pi.
\]
Hence,
\begin{equation}
\label{eq:sharpness-parameter-distance}
  \delta_\pm(\Xi_\delta)=\delta(\Xi_\delta)=\delta.
\end{equation}
Let
\[
  G(t):=\sum_{\ell=0}^{k+1}
  (-1)^\ell\binom{k+1}{\ell}(t-\ell)_+^k.
\]
The function $G$ is supported in $[0,k+1]$ and satisfies $G(t)=t^k$ on $(0,1)$. Writing $z=a\cdot x$,
\[
  v_\delta(x)=\delta^kG\left(\frac{z}{\delta}\right),
  \qquad
  D_a^jv_\delta(x)=\delta^{k-j}G^{(j)}\left(\frac{z}{\delta}\right),
  \quad 0\le j\le k,
\]
where $D_a=a\cdot\nabla$. The projection $a\cdot x$ has distribution $\gamma_1$ under $\gamma$. Since the Gaussian density is uniformly comparable to a positive constant on $[0,1]$, a change of variables gives, for every $0\le j\le k$,
\[
  \|D_a^jv_\delta\|_{L_\gamma^2}^2
  \asymp
  \delta^{2k-2j+1}.
\]
All coordinate derivatives of order $j$ are scalar multiples of the same univariate derivative, and at least one coordinate of $a$ has magnitude at least $d^{-1/2}$. Therefore, \eqref{eq:Hermite-derivative-norm-equivalence} yields, for $m=0,1,\ldots,k$,
\begin{equation}
\label{eq:collision-full-norm-scaling}
  \|v_\delta\|_{\mathcal H_\gamma^m}^2
  \asymp
  \sum_{j=0}^m\delta^{2k-2j+1}
  \asymp
  \delta^{2k-2m+1}.
\end{equation}
Dividing \eqref{eq:collision-full-norm-scaling} at orders $\tau$ and $\sigma$, and using \eqref{eq:sharpness-parameter-distance} proves the first estimate in \eqref{eq:geometric-sharpness-all-orders}. Moreover, $d_\pi((a,b_i),(a,b_j))\ge\pi$ for $i\ne j$, as required.

For the second estimate, take $\Xi_R=\{((1,0,\ldots,0),R)\}\subset\Theta_R$ and $v_R(x)=(x_1-R)_+^k$. For $0\le j\le k$, integration in the first coordinate and the change of variables $y=R(z-R)$ give
\[
  \|\partial_1^jv_R\|_{L_\gamma^2}^2=\left(\frac{k!}{(k-j)!}\right)^2\rho_1(R)R^{-2k+2j-1}\int_0^\infty y^{2k-2j}e^{-y-y^2/(2R^2)}\,dy.
\]
Since $R\ge1$, restricting the integral to $[0,1]$ gives a positive lower bound, while dropping the factor $e^{-y^2/(2R^2)}$ gives a finite upper bound, both independent of $R$. Hence $\|\partial_1^jv_R\|_{L_\gamma^2}^2\asymp\rho_1(R)R^{-2k-1+2j}$. All derivatives involving another coordinate vanish, so \eqref{eq:Hermite-derivative-norm-equivalence} yields, for $m=0,1,\ldots,k$,
\begin{equation}
\label{eq:large-bias-full-norm-scaling}
  \|v_R\|_{\mathcal H_\gamma^m}^2\asymp\rho_1(R)R^{-2k-1}\sum_{j=0}^mR^{2j}\asymp\rho_1(R)R^{-2k-1+2m}.
\end{equation}
Dividing \eqref{eq:large-bias-full-norm-scaling} at orders $\tau$ and $\sigma$ gives
\[
  \frac{\|v_R\|_{\mathcal H_\gamma^\tau}}{\|v_R\|_{\mathcal H_\gamma^\sigma}}\asymp R^{\tau-\sigma}\asymp(1+R)^{\tau-\sigma},
\]
which proves the second estimate in \eqref{eq:geometric-sharpness-all-orders}.
\end{proof}

\subsection{Least squares convergence on \texorpdfstring{$\mathbb R^d$}{Rd}}

In this subsection, we combine approximation, sampling, and inverse estimates to prove convergence under deterministic and random residual sampling and derive the consequences on fixed bounded regions.

The following lemma adapts the spherical approximation results in \cite{LiuMaoXu2025Integral,LiuMaoXu2026DeterministicLS} to growing balls and records the corresponding coefficient bound.

\begin{lemma}
\label{lem:convergence-approximant}
There exists $h_0\in(0,1]$ such that the following holds. Let $R\ge1$, and let $\Xi_{h,R}\subset\Theta_R$ satisfy \eqref{eq:ordinary-parameter-geometry} with uniform constants. If $(1+R)h\le h_0$, then for every $u\in H^{s_k}(B_R)$ there exists $w_h=\sum_{\theta_j\in\Xi_{h,R}}d_j\phi_{\theta_j} \in V_{\Xi_{h,R}}$ such that
\begin{equation}
\label{eq:optimal-sobolev-approximation}
  \|u-w_h\|_{H^q(B_R)}
  \le
  C_q\bigl((1+R)h\bigr)^{s_k-q}
  \|u\|_{H^{s_k}(B_R)},
  \qquad 0\le q<k+1/2,
\end{equation}
and
\begin{equation}
\label{eq:optimal-coefficient-bound}
  \sum_j|d_j|
  \le
  C(1+R)^{1/2}
  \|u\|_{H^{s_k}(B_R)}.
\end{equation}
The constants depend only on $d,k,q,C_b$ and the fixed quasiuniformity constants.
\end{lemma}

\begin{proof}
Set $\Omega_R=B_{R/(1+R)}$ and $g(y)=u((1+R)y)$. Since $R\ge1$, the radius of $\Omega_R$ lies in $[1/2,1)$, where the constants in \autoref{lem:fixed-parameter-critical-sobolev} are uniform. For any extension $U\in H^{s_k}(\mathbb R^d)$ of $u$, Fourier scaling gives
\[
  \|U((1+R)\,\cdot)\|_{H^{s_k}(\mathbb R^d)}^2=(1+R)^{-d}\int_{\mathbb R^d}\bigl(1+(1+R)^2|\xi|^2\bigr)^{s_k}|\widehat U(\xi)|^2\,d\xi.
\]
Using $1+(1+R)^2|\xi|^2\le(1+R)^2(1+|\xi|^2)$ and taking the infimum over extensions, with $s_k-d/2=k+1/2$, yields
\begin{equation}
\label{eq:scaled-target-critical-norm}
  \|g\|_{H^{s_k}(\Omega_R)}
  \le C(1+R)^{k+1/2}\|u\|_{H^{s_k}(B_R)}.
\end{equation}
Map the parameter domain $\Theta_R$ onto a fixed strip of $\mathbb S^d$ by
\begin{equation}
\label{eq:scaled-parameter-map}
  \Pi_R(a,b):=
  \frac{(a,-b/(1+R))}{\sqrt{1+(b/(1+R))^2}}.
\end{equation}
By \autoref{lem:ordinary-euclidean-parameter-family}, choose $\widehat\Xi_{h,R}\subset\Xi_{h,R}$ such that $\Pi_R(\widehat\Xi_{h,R})$ is quasiuniform at spherical scale $h$, together with a finite subset $P_{h,R}\subset\Xi_{h,R}$ whose neurons reproduce every polynomial of degree at most $k$ on $B_R$. Write $\widehat J=\#\widehat\Xi_{h,R}\asymp h^{-d}$. By \autoref{lem:fixed-parameter-critical-sobolev}, there is an approximant $\widetilde g_h=g_h+p_h$, independent of $q$, where $p_h\in\mathcal P_k(\mathbb R^d)$ and
\[
  g_h(y)=\sum_{j=1}^{\widehat J}c_j\bigl[\Pi_R(a_j,b_j)\cdot(y,1)\bigr]_+^k,\qquad (a_j,b_j)\in\widehat\Xi_{h,R}.
\]
In the construction of \autoref{lem:fixed-parameter-critical-sobolev}, the affine arguments associated with the added spherical nodes have a fixed sign on $\Omega_R$. Their contributions are therefore either zero or polynomials of degree at most $k$, and their weighted sum gives $p_h$. Writing $c=(c_j)_{j=1}^{\widehat J}$, estimates \eqref{eq:fixed-parameter-critical-approximation} and \eqref{eq:fixed-parameter-critical-coefficient} yield
\begin{equation}
\label{eq:scaled-spherical-approximation}
  \|g-\widetilde g_h\|_{H^q(\Omega_R)}
  \le C_qh^{s_k-q}\|g\|_{H^{s_k}(\Omega_R)},
  \qquad 0\le q<k+1/2,
\end{equation}
and $\sqrt{\widehat J}\|c\|_2+\|p_h\|_{\mathcal P_k}\le C\|g\|_{H^{s_k}(\Omega_R)}$, where $\|p\|_{\mathcal P_k}:=\sum_{|\alpha|\le k}|p_\alpha|$ for $p(x)=\sum_{|\alpha|\le k}p_\alpha x^\alpha$.

Under $y=x/(1+R)$, each term in $g_h$ satisfies
\[
  \bigl[\Pi_R(a_j,b_j)\cdot(x/(1+R),1)\bigr]_+^k=\bigl((1+R)^2+b_j^2\bigr)^{-k/2}\phi_{a_j,b_j}(x).
\]
Thus Cauchy--Schwarz gives the coefficient bound
\[
  \sum_{j=1}^{\widehat J}|c_j|\bigl((1+R)^2+b_j^2\bigr)^{-k/2}\le(1+R)^{-k}\sqrt{\widehat J}\|c\|_2.
\]
By part~\textup{(ii)} of \autoref{lem:ordinary-euclidean-parameter-family} and \eqref{eq:scaled-poly-ridge-coeff-bound}, $p_h(x/(1+R))$ is represented on $B_R$ using $P_{h,R}$ with coefficient sum at most $C(1+R)^{-k}\|p_h\|_{\mathcal P_k}$. Combining the two representations and merging coefficients at any common parameters gives $w_h\in V_{\Xi_{h,R}}$ with $w_h(x)=\widetilde g_h(x/(1+R))$ on $B_R$ and
\[
  \sum_j|d_j|\le C(1+R)^{-k}\|g\|_{H^{s_k}(\Omega_R)}\le C(1+R)^{1/2}\|u\|_{H^{s_k}(B_R)},
\]
where the last inequality uses \eqref{eq:scaled-target-critical-norm}. This proves \eqref{eq:optimal-coefficient-bound}.

To estimate the error, observe that for $|\alpha|=j$ the same change of variables gives
\[
  \|D^\alpha(u-w_h)\|_{L^2(B_R)}=(1+R)^{d/2-j}\|D^\alpha(g-\widetilde g_h)\|_{L^2(\Omega_R)}.
\]
For $q=n+\theta$ with $n\ge0$ an integer and $0<\theta<1$, the fractional term for $|\alpha|=n$ is
\[
  \int_{B_R}\int_{B_R}\frac{|D^\alpha(u-w_h)(x)-D^\alpha(u-w_h)(x')|^2}{|x-x'|^{d+2\theta}}\,dx\,dx'.
\]
Changing both variables to $x=(1+R)y$ and $x'=(1+R)y'$ contributes the factor $(1+R)^{-2n+2d-(d+2\theta)}=(1+R)^{d-2q}$, from the derivatives, volume elements, and denominator, respectively. Hence the intrinsic characterization on $\Omega_R$ and \eqref{eq:uniform-fractional-norm-equivalence} yield
\[
 \|u-w_h\|_{H^q(B_R)}
   \le C\sum_{j=0}^{n}
   (1+R)^{d/2-j}\|g-\widetilde g_h\|_{H^j(\Omega_R)} +C(1+R)^{d/2-q}
   \|g-\widetilde g_h\|_{H^q(\Omega_R)}.
\]
By \eqref{eq:scaled-target-critical-norm} and \eqref{eq:scaled-spherical-approximation}, each term in the sum is bounded by $C_j((1+R)h)^{s_k-j}\|u\|_{H^{s_k}(B_R)}$, since $(1+R)^{d/2-j}h^{s_k-j}(1+R)^{k+1/2}=((1+R)h)^{s_k-j}$. The fractional term has the same bound with $j=q$. Since $(1+R)h\le1$ and $j\le q$, we have $((1+R)h)^{s_k-j}=((1+R)h)^{s_k-q}((1+R)h)^{q-j}\le((1+R)h)^{s_k-q}$. Summing the finitely many terms proves \eqref{eq:optimal-sobolev-approximation}. For integer $q$, the derivative terms with $0\le j\le q$ give the same conclusion.
\end{proof}

\begin{proof}[Proof of \autoref{thm:convergence}]
We compare the minimizer with an admissible approximant, combining PDE stability with sampling and inverse inequalities to absorb the sampling remainder. The minimizing property of the least squares solution controls the sampled residual, while coefficient and local growth bounds control the Gaussian tails.

Since $2(k-r)>d-1$ and $2(k-r),d\in\mathbb Z$, we have $2(k-r)\ge d$. Hence
\begin{equation}
\label{eq:deterministic-order-identities}
k-r+\frac14>\frac d2, \qquad 0\le S-r\le k-r-1,\qquad k+\frac14<k+\frac12, \qquad k-S+\frac14>0.
\end{equation}
Quasiuniformity with respect to $d_0$ gives $\delta_\pm\lesssim h$ and, for small $h$,
\begin{equation}
\label{eq:deterministic-scale-relations}
\max\{\delta_\pm^{-1},1+R\}\asymp\delta_\pm^{-1},\qquad H\max\{\delta_\pm^{-1},1+R\}\asymp (\log(e/h))^{-p_3}, \qquad H(R+H)\longrightarrow0.
\end{equation}
Hence, for sufficiently small $h$, all patch hypotheses in \autoref{lem:fractional-residual-calculus} and \autoref{lem:deterministic-fractional-sampling} hold.

Because $k+1/4<s_k$, the local and global parts of \autoref{lem:fractional-residual-calculus} apply respectively to $u$ and to every $v\in V_{\Xi_{h,R}}$. The lemma also shows that $\mathcal Au$ and $\mathcal Av$ have continuous representatives on $B_R$, so the sampled residual values are well defined.

Choose the approximant $w_h$ from \autoref{lem:convergence-approximant}. By \eqref{eq:optimal-coefficient-bound} and \eqref{eq:local-growth-u},
\[
  \sum_j|d_j|
  \le C(1+R)^{p_1+1/2}
  \le C (\log(e/h))^{(p_1+1/2)/2}=o(M_h),
\]
so $w_h\in V_{\Xi_{h,R}}(M_h)$. Put $z:=u_h-w_h$ and $\epsilon_h:=w_h-u$. The difference $z$ has an admissible coefficient representation of $\ell^1$ norm at most $2M_h$. Using $\mathcal Au=f$, the triangle inequality and the minimizing property of $u_h$ give
\[
  \|\mathcal Az\|_{X,\omega}\le\|\mathcal Au_h-f\|_{X,\omega}+\|\mathcal Aw_h-f\|_{X,\omega}\le2\|\mathcal Aw_h-f\|_{X,\omega}=2\|\mathcal A\epsilon_h\|_{X,\omega}.
\]
The coefficient bounds in \autoref{ass:solution-local-growth} give $\|\mathcal A\epsilon_h\|_{L_\gamma^2(B_R)}\lesssim(1+R)^{p_2}\|\epsilon_h\|_{H^r(B_R)}$. Together with the upper sampling estimate \eqref{eq:fractional-empirical-upper} and the local operator estimate \eqref{eq:fractional-operator-local}, this yields
\begin{equation}
\label{eq:fractional-empirical-consistency-first}
  \|\mathcal A\epsilon_h\|_{X,\omega}\le C(1+R)^{p_2}\left[\|\epsilon_h\|_{H^r(B_R)}+H^{k-r+1/4}\|\epsilon_h\|_{H^{k+1/4}(B_R)}\right].
\end{equation}
Since $H\lesssim\delta_\pm(\log(e/h))^{-p_3}\lesssim h\le(1+R)h$, substituting \eqref{eq:optimal-sobolev-approximation} at $q=r$ and $q=k+1/4$ for the same $w_h$ gives
\begin{equation}
\label{eq:fractional-empirical-consistency-rate}
  \|\mathcal A\epsilon_h\|_{X,\omega}\le C(1+R)^{p_2}((1+R)h)^{s_k-r}\|u\|_{H^{s_k}(B_R)}.
\end{equation}
The scale relations \eqref{eq:deterministic-scale-relations} give
\begin{equation}
\label{eq:deterministic-absorption}
C(1+R)^{p_2}\bigl(H\max\{\delta_\pm^{-1},1+R\}\bigr)^{k-S+1/4} \lesssim (\log(e/h))^{p_2/2-p_3(k-S+1/4)}\longrightarrow0,
\end{equation}
by the choice of $p_3$ in \autoref{thm:convergence}.
The lower sampling estimate \eqref{eq:fractional-local-sampling} and \eqref{eq:deterministic-absorption} allow us to apply \autoref{lem:fractional-data-norm-transfer} with $D_{R,H}=\|\cdot\|_{X,\omega}$ and $L=2M_h$. Combining the resulting stability estimate with the preceding bounds yields
\begin{equation}
\label{eq:deterministic-principal-z-bound}
\|z\|_{\mathcal H_\gamma^S} \le C H^{-(S-r)}(1+R)^{p_2}((1+R)h)^{s_k-r} \|u\|_{H^{s_k}(B_R)} +CM_h(1+R)^{p_2+k}e^{-cR^2}.
\end{equation}
Using \eqref{eq:local-growth-u}, $H^{-(S-r)}\lesssim \delta_\pm^{-(S-r)}(\log(e/h))^{p_3(S-r)}$, and $1+R\lesssim (\log(e/h))^{1/2}$, the first term is
\[
  C (\log(e/h))^\chi h^{s_k-r}\delta_\pm^{-(S-r)}.
\]

The estimate \eqref{eq:optimal-sobolev-approximation} with $q=S$ controls $u-w_h$ on $B_R$, while parts~(ii) and~(iii) of \autoref{lem:coefficient-tail} control its Gaussian tail. Together with the tail term in \eqref{eq:deterministic-principal-z-bound}, the combined exterior terms are bounded by
\[
C\bigl[M_h(1+R)^{p_2+k} +(1+R)^{p_1+k+1/2}\bigr]e^{-cR^2} \le C(\log(e/h))^\beta h^{c\kappa^2}
\]
for some fixed $\beta\ge0$ and every fixed $0<c<1/4$, with a constant depending on $c$. Since $\kappa^2>4(s_k-S)$, choose $c$ so that $(s_k-S)/\kappa^2<c<1/4$. These terms are then $o(h^{s_k-S})$. Since $\delta_\pm\lesssim h$ and
$\chi\ge(p_1+s_k-S)/2\ge0$, the principal term dominates both these tails
and the interior bound $C(1+R)^{p_1+s_k-S}h^{s_k-S}$.
The triangle inequality proves \eqref{eq:two-scale-convergence-rate}.

If $\delta_\pm\gtrsim h$, then the factor $h^{s_k-r}\delta_\pm^{-(S-r)}$ in \eqref{eq:two-scale-convergence-rate} is bounded by $h^{s_k-S}$.  Also
\[
  J\asymp(1+R)h^{-d},\qquad
  1+R\asymp (\log(e/h))^{1/2},\qquad
  \log(e/h)\asymp1+\log J.
\]
Substitution proves \eqref{eq:width-rate}. The shifted product family satisfies $\delta_\pm\asymp h$ by \autoref{lem:shifted-product-parameters}.
\end{proof}

\begin{proof}[Proof of \autoref{cor:fixed-region-restriction}]
Since $\rho_d$ is bounded below by a positive constant on $\overline\Omega$ and $S$ is an integer, \eqref{eq:Hermite-derivative-norm-equivalence} gives
\[
  \|u^\Omega-u_h^\Omega\|_{H^S(\Omega)}
  \le C_\Omega\|u-u_h\|_{H_\gamma^S(\Omega)}
  \le C_\Omega\|u-u_h\|_{\mathcal H_\gamma^S}.
\]
Combining this bound with \eqref{eq:two-scale-convergence-rate} proves \eqref{eq:fixed-region-two-scale-rate}. If $\delta_\pm\gtrsim h$, applying the same restriction bound to \eqref{eq:width-rate} proves \eqref{eq:fixed-region-one-scale-rates}.
\end{proof}

\begin{proof}[Proof of \autoref{thm:random-physical-least-squares}]
The proof reduces to the deterministic argument on an event that controls the sample counts in every cell and yields sampling inequalities simultaneously for all admissible residuals. Set $H=\delta_\pm(\log(e/h))^{-p_3}$ for the auxiliary partition. The order relations \eqref{eq:deterministic-order-identities} remain unchanged.

Before sampling, fix the partition, inner balls, and patch centers from \autoref{lem:random-partition-occupancy}, and let $\mathcal E_{R,H,N}$ be the event defined there. By \eqref{eq:random-cell-geometry} and the choice of $R$, its cell count satisfies $M_0\asymp\delta_\pm^{-d}(\log(e/h))^{d(p_3+1/2)}$. Since $\delta_\pm\lesssim h$, we have $\log(CM_0/\eta)\lesssim\log(C/(\eta\delta_\pm))$ for sufficiently small $h$. Thus the sample size condition in \autoref{thm:random-physical-least-squares} implies \eqref{eq:balanced-occupancy-sufficient}, hence
\begin{equation}
\label{eq:random-proof-event}
  \mathbb P(\mathcal E_{R,H,N})\ge1-\eta.
\end{equation}
The scale relations \eqref{eq:deterministic-scale-relations} also apply here and imply $H/R\to0$. Thus the patch scale conditions and \eqref{eq:deterministic-absorption} hold for sufficiently small $h$. Consequently, \autoref{lem:fractional-residual-calculus} gives continuous residual representatives. The coefficient ball in $\mathbb R^J$ is compact and the discrete objective is continuous on its image, so an empirical minimizer exists for every realization.

Fix an outcome in $\mathcal E_{R,H,N}$. The sampling estimates \eqref{eq:random-fractional-empirical-upper} and \eqref{eq:random-fractional-physical-sampling} hold simultaneously for every admissible residual.

Choose the same approximant $w_h$ as in the deterministic proof and put $z:=u_{h,N}-w_h$ and $\epsilon_h:=w_h-u$. Then $w_h$ is admissible for small $h$, and $z$ has a representation whose coefficients have $\ell^1$ norm at most $2M_h$. The minimizing property of $u_{h,N}$ gives
\[
\|\mathcal Az\|_{N,\gamma} \le \|\mathcal Au_{h,N}-f\|_{N,\gamma} +\|\mathcal Aw_h-f\|_{N,\gamma} \le2\|\mathcal A\epsilon_h\|_{N,\gamma}.
\]
Since $H\lesssim h\le(1+R)h$, replacing \eqref{eq:fractional-empirical-upper} by \eqref{eq:random-fractional-empirical-upper} in the derivation of \eqref{eq:fractional-empirical-consistency-first}--\eqref{eq:fractional-empirical-consistency-rate} gives the same bounds with $\|\cdot\|_{N,\gamma}$. The lower estimate \eqref{eq:random-fractional-physical-sampling}, \eqref{eq:deterministic-absorption}, and \autoref{lem:fractional-data-norm-transfer}, with $D_{R,H}=\|\cdot\|_{N,\gamma}$ and $L=2M_h$, then yield \eqref{eq:deterministic-principal-z-bound} with $u_h$ replaced by $u_{h,N}$. The estimate \eqref{eq:optimal-sobolev-approximation} with $q=S$ and parts~(ii) and~(iii) of \autoref{lem:coefficient-tail} complete the argument by the triangle inequality. Thus \eqref{eq:random-two-scale-rate} holds throughout $\mathcal E_{R,H,N}$, and \eqref{eq:random-proof-event} proves the claimed probability.

If $\delta_\pm\gtrsim h$, then $\delta_\pm\asymp h$ by \eqref{eq:ordinary-parameter-geometry}, and \eqref{eq:random-width-rate} follows from the same calculation used to obtain \eqref{eq:width-rate}. Since $J\asymp(1+R)h^{-d}$ and $1+R\asymp(\log(e/h))^{1/2}$,
\[
  \delta_\pm^{-d}(\log(e/h))^{d(p_3+1/2)}\asymp J(\log(e/h))^{dp_3+(d-1)/2}.
\]
Using $\log(e/h)\asymp1+\log J$ and $\log(C/(\eta\delta_\pm))\asymp\log(CJ/\eta)$, the sample size condition in \autoref{thm:random-physical-least-squares} is satisfied whenever
\begin{equation}
\label{eq:random-J-sample-threshold}
  N\ge C J(1+\log J)^{dp_3+(d-1)/2}\log(CJ/\eta).
\end{equation}
For fixed $\eta$, the final logarithm contributes one additional power, giving the sufficient sample bound stated in \autoref{thm:random-physical-least-squares}.
\end{proof}

\subsection{Matching lower bound}

In this subsection, we prove the lower bound by restricting to a fixed ball and applying the classical Sobolev width estimate.

\begin{proof}[Proof of \autoref{prop:width-lower-bound}]
Fix a bounded operator $E_c:H^{s_k}(B_{1/4})\longrightarrow H^{s_k}(\mathbb R^d)$ satisfying
\begin{equation}
\label{eq:compact-extension-properties}
  E_cf=f\quad\text{on }B_{1/4},
  \qquad
  \operatorname{supp}(E_cf)\subset B_{1/2},
  \qquad
  \|E_cf\|_{H^{s_k}(\mathbb R^d)}
  \le
  C_E\|f\|_{H^{s_k}(B_{1/4})}.
\end{equation}
Such an operator is obtained from a Sobolev extension and a fixed cutoff \cite{AdamsFournier2003}. Define
\[
  \mathcal U_{s_k}
  :=
  \left\{
    E_cf:
    f\in H^{s_k}(B_{1/4}),
    \ \|f\|_{H^{s_k}(B_{1/4})}\le C_E^{-1}
  \right\}.
\]
Then $\operatorname{supp}u\subset B_{1/2}$ and $\|u\|_{H^{s_k}(\mathbb R^d)}\le1$ for every $u\in\mathcal U_{s_k}$.

On the fixed ball $B_{1/4}$, the Gaussian density is bounded above and below by positive constants. Together with \eqref{eq:Hermite-derivative-norm-equivalence}, this implies that there exists $c>0$ such that
\[
  \|v\|_{\mathcal H_\gamma^S}
  \ge
  c\|v\|_{H^S(B_{1/4})},
  \qquad
  v\in\mathcal H_\gamma^S.
\]

Let $W\subset\mathcal H_\gamma^S$ with $\dim W\le J$. Its restriction $W|_{B_{1/4}}$ also has dimension at most $J$. For $u=E_cf\in\mathcal U_{s_k}$, \eqref{eq:compact-extension-properties} gives $u|_{B_{1/4}}=f$, and therefore
\[
\begin{aligned}
  \sup_{u\in\mathcal U_{s_k}}
  \inf_{w\in W}\|u-w\|_{\mathcal H_\gamma^S}
  &\ge
  c
  \sup_{\substack{f\in H^{s_k}(B_{1/4})\\
      \|f\|_{H^{s_k}(B_{1/4})}\le C_E^{-1}}}
  \inf_{z\in W|_{B_{1/4}}}
  \|f-z\|_{H^S(B_{1/4})}.
\end{aligned}
\]
Kolmogorov and approximation numbers agree for compact embeddings between Hilbert spaces up to an index shift. The estimate for approximation numbers in \cite[Theorem~3.3.4]{EdmundsTriebel1996} therefore gives Kolmogorov widths of order $J^{-(s_k-S)/d}$ for $H^{s_k}(B_{1/4})\hookrightarrow H^S(B_{1/4})$. Taking the infimum over all admissible $W$ proves \eqref{eq:width-lower-bound}.

With dictionaries independent of $u$, the stated uniform hypotheses give the matching upper bounds from \eqref{eq:width-rate} and \eqref{eq:random-width-rate}; the sample event $\mathcal E_{R,H,N}$ controls all residuals simultaneously. For $\mathcal A=I$ and $r=0$, both assumptions hold uniformly with $p_1=p_2=0$, so $k>(d-1)/2$ gives the asserted example.
\end{proof}

\appendix
\section{Auxiliary Lemmas for the Inverse Inequality}
\label{app:inverse-lemmas}

This appendix provides the Hermite coefficient formulas, localization, regularity, and packing estimates used in \autoref{sec:proofs}. Standard facts about Hermite polynomials are recorded in \cite{Thangavelu1993,Szego1975}.

\begin{lemma}
\label{lem:eta-formula}
For every $n\ge k+1$,
\[
  \eta_n(b)=k!\rho_1(b)\sqrt{\frac{(n-k-1)!}{n!}}\,H_{n-k-1}(b).
\]
\end{lemma}

\begin{proof}
Expanding $\rho_1(z-t)=\rho_1(z)e^{tz-t^2/2}$ in powers of $t$ gives the Rodrigues formula in the normalization used here, $H_n(z)\rho_1(z) = \frac{(-1)^n}{\sqrt{n!}}\frac{\dd^n}{\dd z^n}\rho_1(z).$ Integrating by parts $k+1$ times and using Gaussian decay at infinity, we obtain
\[
  \eta_n(b)=\frac{(-1)^n}{\sqrt{n!}}\int_b^\infty(z-b)^k\rho_1^{(n)}(z)\dd z=\frac{k!(-1)^{n-k-1}}{\sqrt{n!}}\rho_1^{(n-k-1)}(b).
\]
Applying the Rodrigues formula once more gives the stated formula.
\end{proof}

The coefficient formula reduces the kernel estimates to sums of Hermite functions. Set
\[
  g_m(u):=\pi^{-1/4}e^{-u^2/2}H_m(\sqrt2\,u).
\]
These are the Hermite functions normalized in $L^2(\mathbb R,dx)$ \cite{PetrushevXu2008}. By \eqref{eq:psi-definition},
\[
  \psi_m(x)=2^{-1/4}g_m(x/\sqrt2).
\]

\begin{lemma}
\label{lem:external-Hermite-localization}
Fix $0<u_0<u_1<\infty$ and an integer $L\ge1$. There exists a constant $C$ such that, for every $\mu\in C_c^\infty((u_0,u_1))$ and $N\ge1$,
\begin{equation}
\label{eq:external-localization}
  \left|
  \sum_{m=0}^{\infty}
  \mu(m/N)\psi_m(x)\psi_m(y)
  \right|
  \le
  C N^{1/2}
  \left(
    \max_{0\le j\le L}
    \norm{\mu^{(j)}}_{L^\infty}
  \right)
  (1+\sqrt N|x-y|)^{-L}
\end{equation}
for all $x,y\in\R$. The constant depends only on $u_0,u_1,L$.
\end{lemma}

\begin{proof}
We apply the localization estimate in \cite[Theorem~1 and Corollary~1, (2.17)]{PetrushevXu2008} in one dimension at $u=x/\sqrt2$ and $v=y/\sqrt2$, using $\psi_m(x)=2^{-1/4}g_m(x/\sqrt2)$. Since $\operatorname{supp}\mu\subset(u_0,u_1)$, the constant depends only on $u_0,u_1,L$, with multiplier dependence bounded by $\max_{0\le j\le L}\|\mu^{(j)}\|_{L^\infty}$. Absorbing the fixed scaling factors gives \eqref{eq:external-localization}.
\end{proof}

For the diagonal lower bound in \autoref{lem:normalized-block-localization}, we also need the following estimate for the Hermite energy in a frequency band.

\begin{lemma}
\label{lem:Hermite-band-lower}
Fix $0<\alpha<\beta$ and $C_b>0$. There exist $c,C>0$, depending only on $\alpha,\beta,C_b$, such that, whenever $R\ge1$, $N\ge C(1+R)^2$, and $|b|\le C_b(1+R)$,
\begin{equation}
\label{eq:Hermite-band-lower}
  \sum_{\alpha N\le m\le \beta N}\psi_m(b)^2
  \ge cN^{1/2}.
\end{equation}
\end{lemma}

\begin{proof}
Denote the kernel of the Hermite partial sum in one dimension by
\[
  \mathcal K_M^{(1)}(u,v)
  :=
  \sum_{j=0}^{M}g_j(u)g_j(v),
\]
so that $\mathcal K_M^{(1)}(u,u) = \sum_{j=0}^{M}g_j(u)^2.$ The superscript $(1)$ indicates the spatial dimension; in the notation of \cite{PetrushevXu2008}, this is precisely $K_{M,1}$.

Since $\psi_m(b)=2^{-1/4}g_m(b/\sqrt2)$,
\[
  \sum_{\alpha N\le m\le \beta N}\psi_m(b)^2
  =
  2^{-1/2}
  \sum_{\alpha N\le m\le \beta N}g_m(b/\sqrt2)^2.
\]
Set $n:=\floor{\beta N}$ and $\lambda:=(1-\alpha/\beta)/2\in(0,1)$. For sufficiently large $C$, $[\floor{(1-\lambda)n},n]\subset[\alpha N,\beta N]$. In one dimension, the estimate in \cite[Lemma~1]{PetrushevXu2008} gives constants $\rho\in(0,1)$ and $c>0$ such that, whenever $|x|\le2\sqrt{2n+1}$,
\[
  \sum_{m=\floor{(1-\lambda)n}}^{n}g_m(x)^2
  \ge
  c\,\mathcal K_{\floor{\rho n}}^{(1)}(x,x).
\]
Set $M=\floor{\rho n}\asymp N$. Since $|b|\le C_b(1+R)$ and $N\ge C(1+R)^2$, choosing $C$ sufficiently large ensures $|b|/\sqrt2\le\sqrt M$. Combining the preceding bound with the estimate in \cite[(2.12)]{PetrushevXu2008} gives
\[
  \sum_{\alpha N\le m\le\beta N}g_m(b/\sqrt2)^2\ge c\mathcal K_M^{(1)}(b/\sqrt2,b/\sqrt2)\ge cM^{1/2}\ge cN^{1/2}.
\]
Combining the preceding estimates proves \eqref{eq:Hermite-band-lower}.
\end{proof}

\begin{proof}[Proof of \autoref{lem:normalized-block-localization}]
By \autoref{lem:eta-formula} and \eqref{eq:psi-definition}, for $n\ge k+1$,
\[
  \widetilde\eta_n(b)
  =
  k!\sqrt{\frac{(n-k-1)!}{n!}}\psi_{n-k-1}(b).
\]
Since $\operatorname{supp}\zeta\subset[1/2,2]$, choosing $N_1\ge4(k+1)$ ensures that $\zeta(n/N)=0$ for $0\le n\le k$ whenever $N\ge N_1$. Applying the coefficient formula to the remaining terms and setting $m=n-k-1\ge0$, we obtain
\[
  \widetilde m_N(\theta,\theta')
  =
  (k!)^2
  \sum_{m=0}^{\infty}
  \zeta((m+k+1)/N)
  \frac{m!}{(m+k+1)!}
  (a\cdot a')^{m+k+1}
  \psi_m(b)\psi_m(b').
\]

Set $\rho:=|a\cdot a'|$. If $\rho=0$, the kernel vanishes. Otherwise, let $b_*:=\operatorname{sgn}(a\cdot a')b'$. Since $H_m(-x)=(-1)^m H_m(x)$ and $\rho_1(-x)=\rho_1(x)$, $\psi_m(-x)=(-1)^m\psi_m(x).$ Hence, $(a\cdot a')^{m+k+1}\psi_m(b') = (a\cdot a')^{k+1} \rho^m\psi_m(b_*),$ and therefore
\[
  \widetilde m_N(\theta,\theta')
  =
  (k!)^2(a\cdot a')^{k+1}
  \sum_{m=0}^{\infty}
  \zeta((m+k+1)/N)
  \frac{m!}{(m+k+1)!}
  \rho^m
  \psi_m(b)\psi_m(b_*).
\]

Define
\[
  \mu_{N,\rho}(t)
  :=
  \zeta(t+(k+1)/N)
  \prod_{j=1}^{k+1}(t+j/N)^{-1}
  \rho^{Nt},
  \qquad t\ge0.
\]
At $t=m/N$, $\frac{m!}{(m+k+1)!} = N^{-k-1} \prod_{j=1}^{k+1}(m/N+j/N)^{-1},$ so
\[
  \zeta((m+k+1)/N)
  \frac{m!}{(m+k+1)!}
  \rho^m
  =
  N^{-k-1}\mu_{N,\rho}(m/N).
\]
Thus,
\[
  \widetilde m_N(\theta,\theta')
  =
  (k!)^2(a\cdot a')^{k+1} N^{-k-1}
  \sum_{m=0}^{\infty}
  \mu_{N,\rho}(m/N)
  \psi_m(b)\psi_m(b_*).
\]

We bound the derivatives of $\mu_{N,\rho}$. For $N\ge N_1$, its support lies in $[1/4,2]$, a fixed compact interval away from $0$. Let $\lambda:=-N\log\rho\ge0$, so that $\rho^{Nt}=e^{-\lambda t}$. On the support, the function
\[
  \zeta(t+(k+1)/N)\prod_{j=1}^{k+1}(t+j/N)^{-1}
\]
and all of its derivatives up to any fixed order are bounded uniformly in $N$. For $0\le j\le L$ and every fixed nonnegative integer $M$, we have $\sup_{\lambda\ge0}(1+\lambda)^M\lambda^j e^{-c\lambda}<\infty$. Leibniz's rule therefore gives
\[
  \max_{0\le j\le L}
  \norm{\partial_t^j\mu_{N,\rho}}_{L^\infty}
  \le
  C_{L,M,k,\zeta}(1+\lambda)^{-M}.
\]

Let $\varphi:=\arccos\rho\in[0,\pi/2]$. Since $-\log\rho\ge1-\rho$ and $1-\cos\varphi\ge c\varphi^2$, $\lambda=N(-\log\rho)\ge cN\varphi^2.$ Taking $M=\lceil L/2\rceil$ gives $(1+\lambda)^{-M} \le C_L(1+\sqrt N\varphi)^{-L}.$ Consequently,
\[
  \max_{0\le j\le L}
  \norm{\partial_t^j\mu_{N,\rho}}_{L^\infty}
  \le
  C_L(1+\sqrt N\varphi)^{-L}.
\]

Applying \autoref{lem:external-Hermite-localization} with $x=b$ and $y=b_*$ yields
\[
\begin{aligned}
  |\widetilde m_N(\theta,\theta')|
  &\le
  C N^{-k-1}N^{1/2}
  (1+\sqrt N\varphi)^{-L}
  (1+\sqrt N|b-b_*|)^{-L}\\
  &=
  C N^{-k-1/2}
  (1+\sqrt N\varphi)^{-L}
  (1+\sqrt N|b-b_*|)^{-L}.
\end{aligned}
\]
If $a\cdot a'>0$, then $b_*=b'$ and $\varphi=\rho_{\Sph^{d-1}}(a,a')$, so $\varphi+|b-b_*|=d_0(\theta,\theta')$. If $a\cdot a'<0$, then $b_*=-b'$ and $\varphi=\rho_{\Sph^{d-1}}(a,-a')$, so $\varphi+|b-b_*|=d_\pi(\theta,\theta')$. Using $(1+u)^{-L}(1+v)^{-L} \le (1+u+v)^{-L},$ we obtain \eqref{eq:normalized-block-localization}.

It remains to prove the diagonal lower bound. For $\theta=\theta'$,
\[
  \widetilde m_N(\theta,\theta)
  =
  (k!)^2
  \sum_{m=0}^{\infty}
  \zeta((m+k+1)/N)
  \frac{m!}{(m+k+1)!}
  \psi_m(b)^2.
\]
Since $\zeta(1)=1$, we have $\zeta\ge c>0$ on a fixed interval $[\alpha_0,\beta_0]\subset(1/2,2)$. Whenever $\alpha_0N\le m+k+1\le\beta_0N$, $\zeta((m+k+1)/N)\ge c.$ Moreover, $m\asymp N$ on this interval, so $\frac{m!}{(m+k+1)!} = \prod_{j=1}^{k+1}(m+j)^{-1} \ge cN^{-k-1}.$ Hence,
\[
  \widetilde m_N(\theta,\theta)
  \ge
  cN^{-k-1}
  \sum_{\alpha_0N\le m+k+1\le \beta_0N}
  \psi_m(b)^2.
\]
Choose $\alpha_0<\alpha_1<\beta_1<\beta_0.$ Since $k$ is fixed, for $N\ge(k+1)/(\beta_0-\beta_1)$, $[\alpha_1N,\beta_1N] \subset [\alpha_0N-k-1,\beta_0N-k-1].$ Therefore,
\[
  \sum_{\alpha_0N\le m+k+1\le \beta_0N}\psi_m(b)^2
  \ge
  \sum_{\alpha_1N\le m\le \beta_1N}\psi_m(b)^2.
\]
If $N\ge C(1+R)^2$ and $\theta\in\Theta_R$, then $|b|\le C_b(1+R)$. After enlarging the fixed initial frequency if necessary, \autoref{lem:Hermite-band-lower} applied to $[\alpha_1,\beta_1]$ gives
\[
  \sum_{\alpha_0N\le m+k+1\le \beta_0N}\psi_m(b)^2
  \ge
  cN^{1/2}.
\]
Thus,
\[
  \widetilde m_N(\theta,\theta)
  \ge
  cN^{-k-1}N^{1/2}
  =
  cN^{-k-1/2}.
\]

Finally, let $\theta=(a,b)\in\Theta_R$ and $0\le s<k+1/2$. The $L^2(\gamma)$ Hermite expansion derived in \autoref{subsec:inverse-parameters} gives
\[
  \langle\widetilde\phi_\theta,H_\alpha\rangle_{L^2(\gamma)}
  =\widetilde\eta_n(b)\sqrt{\frac{n!}{\alpha!}}\,a^\alpha,
  \qquad |\alpha|=n.
\]
Since $|a|=1$, the multinomial formula yields
\[
  \sum_{|\alpha|=n}|\langle\widetilde\phi_\theta,H_\alpha\rangle_{L^2(\gamma)}|^2
  =|\widetilde\eta_n(b)|^2\sum_{|\alpha|=n}\frac{n!}{\alpha!}a^{2\alpha}
  =|\widetilde\eta_n(b)|^2.
\]
Grouping the terms in \eqref{eq:Hermite-Sobolev-definition} by $|\alpha|=n$, we therefore obtain
\[
  \|\widetilde\phi_\theta\|_{\mathcal H_\gamma^s}^2
  =\sum_{n\ge0}(1+n)^s|\widetilde\eta_n(b)|^2,
\]
where the nonnegative series may initially be infinite.

Insert the dyadic partition \eqref{eq:dyadic-partition} into this series and interchange the nonnegative sums. Since $\zeta_0(0)=1$ and $\zeta_0(n)=0$ for integers $n\ge1$, the window at low frequencies contributes $|\widetilde\eta_0(b)|^2$. Together with the finitely many blocks with $N<N_1$, this contribution is bounded by $C_{R,\theta,s}$. On each remaining block, $\zeta(n/N)\ne0$ implies $N/2\le n\le2N$, so $(1+n)^s\le3^sN^s$. The identity $\widetilde m_N(\theta,\theta)=\sum_{n\ge0}\zeta(n/N)|\widetilde\eta_n(b)|^2$ and \eqref{eq:normalized-block-localization} with $\theta'=\theta$ consequently give
\[
  \sum_{n\ge0}(1+n)^s\zeta(n/N)|\widetilde\eta_n(b)|^2
  \le 3^sN^s\widetilde m_N(\theta,\theta)
  \le CN^{s-k-1/2}.
\]
Summing over the dyadic scales yields
\[
  \|\widetilde\phi_\theta\|_{\mathcal H_\gamma^s}^2
  \le C_{R,\theta,s}+C\sum_{\substack{N=2^j\\N\ge N_1}}N^{s-k-1/2}<\infty,
\]
because $\sum_{j\ge0}2^{j(s-k-1/2)}$ is a convergent geometric series when $s<k+1/2$. Finally, $\phi_\theta=\rho_1(b)^{1/2}\widetilde\phi_\theta$ with $\rho_1(b)>0$, so $\phi_\theta\in\mathcal H_\gamma^s$.
\end{proof}

To pass from the kernel bounds to the matrix estimate in \autoref{lem:center-dominance}, we use parameter separation to bound the absolute row sums.

\begin{lemma}
\label{lem:packing}
Let $0<\rho\le\delta_\pm(\Xi)$, $L>d$, and $Q>0$. Then
\[
  \sup_i\sum_{j\ne i}(1+Qd_\pm(\theta_i,\theta_j))^{-L}\le C_{d,C_b,L}(Q\rho)^{-L}.
\]
\end{lemma}

\begin{proof}
We first prove that, for every $\theta_i\in\Xi$ and $r>0$,
\begin{equation}
\label{eq:packing-number}
  \#\{j:d_\pm(\theta_i,\theta_j)<r\}
  \le C\left(1+\frac{r}{\rho}\right)^d.
\end{equation}
Here $C$ depends only on $d,C_b$.

With respect to the product measure $d\sigma(a)\,db$ on $\Theta_R$, where $d\sigma$ is surface measure on $\mathbb S^{d-1}$, every ball of radius $t>0$ in the $d_0$ metric has volume comparable to
\[
  V_R(t):=\min\{t,1\}^{d-1}\min\{t,1+R\},
\]
uniformly in its center and $R\ge1$. Indeed, such a ball contains the product of a spherical ball and a truncated bias interval of radii $t/2$ and is contained in the corresponding product of radii $t$. Spherical ball measures are comparable to $\min\{t,1\}^{d-1}$, including counting measure on $\mathbb S^0$ when $d=1$. The truncated interval lengths are comparable to $\min\{t,1+R\}$, including at the endpoints of the bias interval.

Fix $\theta_i\in\Xi$. For each $j$ with $d_\pm(\theta_i,\theta_j)<r$, choose $\varepsilon_j\in\{-1,1\}$ so that $d_0(\theta_i,\varepsilon_j\theta_j)<r$, where $-(a,b):=(-a,-b)$. For distinct selected indices,
\[
  d_0(\varepsilon_j\theta_j,\varepsilon_\ell\theta_\ell)\ge d_\pm(\theta_j,\theta_\ell)\ge\rho.
\]
In the $d_0$ metric, the balls of radius $\rho/3$ about these representatives are pairwise disjoint and lie in the ball centered at $\theta_i$ with radius $r+\rho/3$. Comparing volumes gives
\[
  \#\{j:d_\pm(\theta_i,\theta_j)<r\}
  \le C\frac{V_R(r+\rho/3)}{V_R(\rho/3)}
  \le C\left(1+\frac{3r}{\rho}\right)^d
  \le C\left(1+\frac r\rho\right)^d,
\]
which proves \eqref{eq:packing-number}.

For the weighted sum, decompose the indices into annuli
\[
  2^\ell \rho\le d_\pm(\theta_i,\theta_j)<2^{\ell+1}\rho,
  \qquad \ell=0,1,2,\ldots .
\]
By \eqref{eq:packing-number}, the $\ell$th annulus contains at most $C2^{\ell d}$ points. Hence,
\[
  \sum_{j\ne i}
  (1+Qd_\pm(\theta_i,\theta_j))^{-L}
  \le
  C\sum_{\ell\ge0}
  2^{\ell d}(Q2^\ell \rho)^{-L}
  \le
  C(Q\rho)^{-L},
\]
where the last inequality uses $L>d$.
\end{proof}

\section{Auxiliary Results for Convergence Analysis on \texorpdfstring{$\mathbb R^d$}{Rd}}
\label{app:convergence-lemmas}

This appendix collects the approximation, parameter, sampling, operator, and tail results used in the convergence proof on $\mathbb R^d$.

\subsection{Approximation construction}

In this subsection, we establish Sobolev approximation on growing balls with explicit coefficient bounds and construct parameter sets with the required coverage and separation.

Let $d\nu=dS/|\mathbb S^d|$ be normalized surface measure on $\mathbb S^d$ and set $\sigma_k(t):=t_+^k$. Let $\Pi_m$ be the orthogonal projection in $L^2(\mathbb S^d,\nu)$ onto the spherical harmonics of degree $m$. We equip $H^s(\mathbb S^d)$ with the norm
\[
  \|F\|_{H^s(\mathbb S^d)}^2:=\sum_{m=0}^{\infty}(1+m)^{2s}\|\Pi_m F\|_{L^2(\nu)}^2,\qquad s\ge0.
\]
The multiplier $\widehat{\sigma_k}(m)$ is defined by
\[
  \int_{\mathbb S^d}\sigma_k(\theta\cdot\eta)Y(\eta)\,d\nu(\eta)
  =\widehat{\sigma_k}(m)Y(\theta),
  \qquad \theta\in\mathbb S^d,
\]
for every spherical harmonic $Y$ of degree $m$.

For $C_b>1$, define the spherical strip
\begin{equation}
\label{eq:spherical-strip-definition}
  \Gamma_{C_b}:=\left\{\frac{(a,t)}{\sqrt{1+t^2}}:a\in\mathbb S^{d-1},\ |t|\le C_b\right\}\subset\mathbb S^d.
\end{equation}
All spherical fill and separation distances below are measured in the geodesic metric.

\begin{lemma}
\label{lem:fixed-parameter-critical-sobolev}
Fix $C_b>1$ and $1/2\le\varrho<1$. Suppose $\Theta_h= \{\vartheta_j:1\le j\le J\} \subset\Gamma_{C_b}$ is quasiuniform on $\Gamma_{C_b}$ at spherical scale $h$; in particular, $J\asymp h^{-d}$. Then there exists a linear approximation operator, independent of the target order $q$,
\[
  \mathcal Q_h:H^{s_k}(B_\varrho)
  \longrightarrow
  \operatorname{span}\{\sigma_k(\vartheta_j\cdot(x,1))\}_{j=1}^{J}
  +\mathcal P_k(\mathbb R^d),
\]
such that, for every $f\in H^{s_k}(B_\varrho)$, writing $\mathcal Q_hf=f_h+p_h$, with $p_h\in\mathcal P_k(\mathbb R^d)$, one has
\begin{equation}
\label{eq:fixed-parameter-critical-approximation}
  \|f-\mathcal Q_hf\|_{H^q(B_\varrho)}
  \le C_q h^{s_k-q}\|f\|_{H^{s_k}(B_\varrho)},
  \qquad 0\le q<k+\frac12,
\end{equation}
for every fixed real $q$ in this range. If $c=(c_j)_{j=1}^{J}$ denotes the coefficient vector of $f_h$, then
\begin{equation}
\label{eq:fixed-parameter-critical-coefficient}
  \sqrt J\|c\|_2
  +\|p_h\|_{\mathcal P_k}
  \le C\|f\|_{H^{s_k}(B_\varrho)}.
\end{equation}
Here $\|p\|_{\mathcal P_k}:=\sum_{|\alpha|\le k}|p_\alpha|$ for $p(x)=\sum_{|\alpha|\le k}p_\alpha x^\alpha$. The constants are uniform in $1/2\le\varrho<1$, $h$, and $f$.
\end{lemma}

\begin{proof}
We first represent $f$ by a function on $\mathbb S^d$. The construction in \cite[Section~5.1]{LiuMaoXu2026DeterministicLS} defines a bounded linear map $f\mapsto F\in H^{s_k}(\mathbb S^d)$ with the required symmetry $F(-\eta)=(-1)^{k+1}F(\eta)$ and
\begin{equation}
\label{eq:spherical-lift-bound}
  f(x)=|(x,1)|^kF((x,1)/|(x,1)|),\quad x\in B_\varrho,\qquad \|F\|_{H^{s_k}(\mathbb S^d)}\le C\|f\|_{H^{s_k}(B_\varrho)}.
\end{equation}
The constant $C$ can be chosen independently of $1/2\le\varrho<1$: extension from these balls has a uniform bound by scaling, and the subsequent construction on the sphere is fixed.

Extend $\Theta_h$ to a maximal separated set $\widetilde\Theta_h=\{\widetilde\vartheta_j\}_{j=1}^{\widetilde J}$ by adding points only from $\mathbb S^d\setminus\operatorname{int}(\Gamma_{C_b})$, with separation at least a fixed positive multiple of $h$ chosen below the original separation. Maximality implies that every point outside the strip lies within $O(h)$ of a selected node; inside the strip, the original nodes already provide this covering. Thus the completed set is quasiuniform on $\mathbb S^d$, and spherical packing and covering estimates yield $\widetilde J\asymp J\asymp h^{-d}$. Number the nodes so that $\widetilde\vartheta_j=\vartheta_j$ for $1\le j\le J$.

It suffices to consider sufficiently small $h$: when $h$ is bounded below, the zero operator satisfies both conclusions by the uniform embedding $H^{s_k}(B_\varrho)\hookrightarrow H^q(B_\varrho)$. Fix a sufficiently small constant $\varepsilon>0$ and take $N=\lfloor\varepsilon h^{-1}\rfloor\ge k+1$. The spherical cubature construction in \cite{MhaskarNarcowichWard2001,MhaskarNarcowichWard2002Corrigendum} provides nonnegative weights $\tau_j$ such that
\[
  \sum_{j=1}^{\widetilde J}\tau_j P(\widetilde\vartheta_j)=\int_{\mathbb S^d}P(\eta)\,d\nu(\eta)
\]
for every spherical polynomial $P$ of degree at most $2N$. Positivity and exactness bound each weight by the degree-$N$ spherical Christoffel function, yielding $\tau_j\lesssim N^{-d}\lesssim h^d$. Following \cite[Lemma~3.2 and Theorem~4.1]{LiuMaoXu2025Integral}, set
\[
  G_N:=\sum_{\substack{0\le m\le N\\\widehat{\sigma_k}(m)\ne0}}\widehat{\sigma_k}(m)^{-1}\Pi_m F
\]
and define
\begin{equation}
\label{eq:fixed-operator-coefficient-formula}
  \lambda_j(F):=\tau_jG_N(\widetilde\vartheta_j),\qquad Q_hF(\eta):=\sum_{j=1}^{\widetilde J}\lambda_j(F)\sigma_k(\widetilde\vartheta_j\cdot\eta).
\end{equation}
The construction \eqref{eq:fixed-operator-coefficient-formula} is linear in $F$ and independent of $q$.

We first bound the coefficients. The formula in \cite[equation~(3.10)]{LiuMaoXu2025Integral} shows that $\widehat{\sigma_k}(m)=0$ precisely when $m>k$ and $m-k$ is even. For $m\ge k+1$ with $m-k$ odd, applying the Gamma duplication formula to the coefficient expression in \cite[equation~(3.10) and the proof of Lemma~3.1]{LiuMaoXu2025Integral} gives
\[
  |\widehat{\sigma_k}(m)|\asymp\frac{\Gamma((m-k)/2)}{\Gamma((m+d+k+1)/2)}\asymp(1+m)^{-s_k}.
\]
The last comparison follows from the Gamma ratio asymptotics, since the two arguments differ by $s_k=(d+2k+1)/2$. Absorbing the finitely many nonzero coefficients with $m\le k$ into the constants gives $|\widehat{\sigma_k}(m)|^{-2}\le C(1+m)^{2s_k}$ whenever $\widehat{\sigma_k}(m)\ne0$. Orthogonality and the definition of the Sobolev norm therefore imply
\[
  \|G_N\|_{L^2(\nu)}^2=\sum_{\substack{0\le m\le N\\\widehat{\sigma_k}(m)\ne0}}|\widehat{\sigma_k}(m)|^{-2}\|\Pi_m F\|_{L^2(\nu)}^2\le C\|F\|_{H^{s_k}(\mathbb S^d)}^2.
\]
Since $0\le\tau_j\le Ch^d$ and cubature is exact for $|G_N|^2$, we obtain
\[
  \|\lambda(F)\|_2^2\le Ch^d\sum_{j=1}^{\widetilde J}\tau_j|G_N(\widetilde\vartheta_j)|^2=Ch^d\|G_N\|_{L^2(\nu)}^2\le Ch^d\|F\|_{H^{s_k}(\mathbb S^d)}^2.
\]
Together with $\widetilde J\asymp h^{-d}$, this proves
\begin{equation}
\label{eq:spherical-critical-coefficient-bound}
  \sqrt{\widetilde J}\|\lambda(F)\|_2\le C\|F\|_{H^{s_k}(\mathbb S^d)}.
\end{equation}

To estimate $F-Q_hF$, fix $0\le\ell\le N$ and let $Y$ be any spherical harmonic of degree $\ell$. The Funk--Hecke formula \cite[equations~(3.6)--(3.8)]{LiuMaoXu2025Integral}, in the normalization above, gives
\[
  \int_{\mathbb S^d}Q_hF(\eta)Y(\eta)\,d\nu(\eta)=\widehat{\sigma_k}(\ell)\sum_{j=1}^{\widetilde J}\tau_jG_N(\widetilde\vartheta_j)Y(\widetilde\vartheta_j).
\]
If $\widehat{\sigma_k}(\ell)\ne0$, exactness for the polynomial $G_NY$, whose degree is at most $2N$, and orthogonality of distinct harmonic degrees yield
\[
  \sum_{j=1}^{\widetilde J}\tau_jG_N(\widetilde\vartheta_j)Y(\widetilde\vartheta_j)=\int_{\mathbb S^d}G_N(\eta)Y(\eta)\,d\nu(\eta)=\widehat{\sigma_k}(\ell)^{-1}\int_{\mathbb S^d}F(\eta)Y(\eta)\,d\nu(\eta).
\]
If $\widehat{\sigma_k}(\ell)=0$, then $\ell>k$ and $\ell-k$ is even. In this case $Y(-\eta)=(-1)^\ell Y(\eta)$ and the symmetry of $F$ make $FY$ odd, so $\int_{\mathbb S^d}FY\,d\nu=0$; the Funk--Hecke formula also gives $\int_{\mathbb S^d}(Q_hF)Y\,d\nu=0$. Thus, in both cases, $\int_{\mathbb S^d}(F-Q_hF)Y\,d\nu=0$. Since $Y$ was arbitrary in its harmonic space,
\[
  \Pi_\ell(Q_hF)=\Pi_\ell F,\qquad 0\le\ell\le N.
\]

The identities for $0\le\ell\le N$ leave only the two tails:
\[
  \|F-Q_hF\|_{H^q(\mathbb S^d)}^2\le 2\sum_{m>N}(1+m)^{2q}\|\Pi_m F\|_{L^2(\nu)}^2+2\sum_{m>N}(1+m)^{2q}\|\Pi_m(Q_hF)\|_{L^2(\nu)}^2.
\]
For the first tail, write $(1+m)^{2q}=(1+m)^{2s_k}(1+m)^{-2(s_k-q)}$ and use $m>N$ to obtain
\[
  \sum_{m>N}(1+m)^{2q}\|\Pi_m F\|_{L^2(\nu)}^2\le(1+N)^{-2(s_k-q)}\|F\|_{H^{s_k}(\mathbb S^d)}^2.
\]
For the network tail, the summation-by-parts estimate in \cite[equations~(4.16)--(4.18)]{LiuMaoXu2025Integral}, applied to the quasiuniform completed nodes, gives
\[
  \sum_{m>N}(1+m)^{2q}\|\Pi_m(Q_hF)\|_{L^2(\nu)}^2
  \le C_qN^{-(2k+1-2q)}\|\lambda(F)\|_2^2.
\]
The series in that estimate is $\sum_{m>N}m^{2q-2k-2}$, which converges for $q<k+1/2$; see also \cite[Theorem~6]{LiuMaoXu2026DeterministicLS} for the approximation result in this range. Since $N\asymp h^{-1}$, $\widetilde J\asymp h^{-d}$, and $2s_k=d+2k+1$, \eqref{eq:spherical-critical-coefficient-bound} gives $N^{-(2k+1-2q)}\|\lambda(F)\|_2^2\le C_qh^{2(s_k-q)}\|F\|_{H^{s_k}(\mathbb S^d)}^2$. Combining the two tails and taking square roots yields
\begin{equation}
\label{eq:same-spherical-operator-all-orders}
  \|F-Q_hF\|_{H^q(\mathbb S^d)}\le C_qh^{s_k-q}\|F\|_{H^{s_k}(\mathbb S^d)},\qquad 0\le q<k+\frac12.
\end{equation}

Define $\mathcal Q_hf(x):=|(x,1)|^kQ_hF((x,1)/|(x,1)|)$ for $x\in B_\varrho$. Positive homogeneity gives, for each node,
\[
  |(x,1)|^k\sigma_k\!\left(\widetilde\vartheta_j\cdot\frac{(x,1)}{|(x,1)|}\right)=\sigma_k(\widetilde\vartheta_j\cdot(x,1)).
\]
The original nodes therefore give $f_h(x)=\sum_{j=1}^{J}\lambda_j(F)\sigma_k(\vartheta_j\cdot(x,1))$, with coefficient vector $c=(\lambda_1(F),\ldots,\lambda_J(F))$. A nonpolar added node has $\widetilde\vartheta=(a,t)/\sqrt{1+t^2}$ with $|t|\ge C_b>1>\varrho$, and its contribution is
\[
  \sigma_k(\widetilde\vartheta\cdot(x,1))=(1+t^2)^{-k/2}(a\cdot x+t)_+^k.
\]
For $x\in B_\varrho$, $|a\cdot x|<\varrho$. If $t>0$, then $a\cdot x+t>0$, so the positive part can be removed and the expression is a polynomial of degree at most $k$; if $t<0$, the expression is zero. The north and south poles give the constants $1$ and $0$, respectively. Summing the polynomial expressions of the added nodes defines $p_h\in\mathcal P_k(\mathbb R^d)$ and gives $\mathcal Q_hf=f_h+p_h$ on $B_\varrho$.

The representations of $f$ and $\mathcal Q_hf$ give
\[
  f(x)-\mathcal Q_hf(x)=|(x,1)|^k(F-Q_hF)((x,1)/|(x,1)|),\qquad x\in B_\varrho.
\]
The coordinate map and multiplier are smooth on the fixed ball $B_1$, and the inverse coordinate map on its spherical cap has denominator $\eta_{d+1}\ge1/\sqrt2$. The Sobolev estimates for coordinate changes and multiplication in \cite{AdamsFournier2003} bound this pullback at integer orders. The interpolation theorem in \cite[Theorem~4.1.2]{BerghLofstrom1976} extends these bounds to fractional orders on $B_1$; restriction to $B_\varrho$ preserves uniformity in $\varrho$. Combining this with \eqref{eq:same-spherical-operator-all-orders} and \eqref{eq:spherical-lift-bound} gives
\[
  \|f-\mathcal Q_hf\|_{H^q(B_\varrho)}\le C_q\|F-Q_hF\|_{H^q(\mathbb S^d)}\le C_qh^{s_k-q}\|f\|_{H^{s_k}(B_\varrho)}.
\]
This proves \eqref{eq:fixed-parameter-critical-approximation}.

It remains to bound the coefficients after this decomposition. Since $c$ is a subvector of $\lambda(F)$ and $J\le\widetilde J$, one has $\sqrt J\|c\|_2\le\sqrt{\widetilde J}\|\lambda(F)\|_2$. For an added node written as $(v,s)\in\mathbb S^d$, every nonzero polynomial contribution is $(v\cdot x+s)^k$. The multinomial formula and the definition of the polynomial coefficient norm give
\[
  \|(v\cdot x+s)^k\|_{\mathcal P_k}=\left(|s|+\sum_{i=1}^d|v_i|\right)^k\le(d+1)^{k/2},
\]
where the last inequality uses $|v|^2+s^2=1$. The triangle inequality and Cauchy--Schwarz now show
\[
  \|p_h\|_{\mathcal P_k}\le C\sum_{j=J+1}^{\widetilde J}|\lambda_j(F)|\le C\sqrt{\widetilde J}\|\lambda(F)\|_2.
\]
Combining these bounds with \eqref{eq:spherical-critical-coefficient-bound} and \eqref{eq:spherical-lift-bound} proves
\[
  \sqrt J\|c\|_2+\|p_h\|_{\mathcal P_k}\le C\sqrt{\widetilde J}\|\lambda(F)\|_2\le C\|f\|_{H^{s_k}(B_\varrho)},
\]
which is \eqref{eq:fixed-parameter-critical-coefficient}. All constructions are linear in $f$ and independent of $q$, so the same approximation operator satisfies the estimates for every fixed $q$ in the stated range.
\end{proof}

We next represent the polynomial term in \autoref{lem:fixed-parameter-critical-sobolev} by $\mathrm{ReLU}^k$ networks, retaining coefficient control under small parameter perturbations.

\begin{lemma}
\label{lem:polynomial-ridge-reproduction}
There exist pairwise distinct directions $\bar a_1,\ldots,\bar a_L\in\mathbb S^{d-1}$, with $L=L(d,k)$, distinct constants $2<\beta_0,\ldots,\beta_k<C_b$, and $\varepsilon_0>0$ with the following property. Suppose $a'_{\ell r}\in\mathbb S^{d-1}$ and $\beta'_{\ell r}\in\mathbb R$ satisfy $\rho_{\mathbb S^{d-1}}(a'_{\ell r},\bar a_\ell)\le\varepsilon_0$ and $|\beta'_{\ell r}-\beta_r|\le\varepsilon_0$ for every $\ell,r$. For every $R\ge1$ and $p\in\mathcal P_k$, there exist coefficients $d_{\ell r}$ such that
\[
  p(x/(1+R))=\sum_{\ell=1}^{L}\sum_{r=0}^{k}d_{\ell r}(a'_{\ell r}\cdot x+\beta'_{\ell r}(1+R))_+^k,\qquad x\in B_R,
\]
and
\begin{equation}
\label{eq:scaled-poly-ridge-coeff-bound}
  \sum_{\ell,r}|d_{\ell r}|\le C(1+R)^{-k}\|p\|_{\mathcal P_k}.
\end{equation}
The constants depend only on $d,k,C_b$.
\end{lemma}

\begin{proof}
Choose finitely many distinct directions $\bar a_\ell$ so that $(\bar a_\ell\cdot y)^m$ span the homogeneous polynomials of degree $m$ for every $0\le m\le k$; for $d=1$, take $L=1$ and $\bar a_1=1$. For distinct $\beta_r\in(2,C_b)$, the matrix $\bigl(\binom{k}{m}\beta_r^{k-m}\bigr)_{0\le m,r\le k}$ is an invertible Vandermonde matrix up to row scaling and reversal. Thus $(\bar a_\ell\cdot y+\beta_r)^k$ span $\mathcal P_k$. Their monomial coefficient matrix has a nonsingular square submatrix, whose inverse remains uniformly bounded under sufficiently small independent perturbations of all parameters. The perturbed coefficient matrix therefore has a uniformly bounded right inverse. Choose $\varepsilon_0$ small enough that also $\beta'_{\ell r}>3/2$.

Applying this right inverse to $(1+R)^{-k}p(y)$ gives $(1+R)^{-k}p(y)=\sum_{\ell,r}d_{\ell r}(a'_{\ell r}\cdot y+\beta'_{\ell r})^k$ and \eqref{eq:scaled-poly-ridge-coeff-bound}. Substitute $y=x/(1+R)$ and multiply by $(1+R)^k$. Since $a'_{\ell r}\cdot x+\beta'_{\ell r}(1+R)>0$ on $B_R$, the powers equal their positive parts, proving the stated representation.

The same network functions also represent every $p\in\mathcal P_k$ on $B_R$ with coefficient sum at most $C\max_{|\alpha|\le k}\|D^\alpha p\|_{L^\infty(B_R)}$. To see this, replace $p(y)$ by $p((1+R)y)$ in the construction and use $(1+R)^{-k}\|p((1+R)\,\cdot)\|_{\mathcal P_k}=\sum_{|\alpha|\le k}(1+R)^{|\alpha|-k}|D^\alpha p(0)|/\alpha!\le C\max_{|\alpha|\le k}\|D^\alpha p\|_{L^\infty(B_R)}$.
\end{proof}

We now select the parameters required by Lemmas~\ref{lem:fixed-parameter-critical-sobolev} and~\ref{lem:polynomial-ridge-reproduction} from a prescribed dictionary satisfying \eqref{eq:ordinary-parameter-geometry}.

\begin{lemma}
\label{lem:ordinary-euclidean-parameter-family}
There exists $h_0\in(0,1]$, depending only on $d,k,C_b$ and the fixed quasiuniformity constants in \eqref{eq:ordinary-parameter-geometry}, such that, for every $R\ge1$ and $0<h\le h_0$, any dictionary $\Xi_{h,R}\subset\Theta_R$ satisfying \eqref{eq:ordinary-parameter-geometry} admits the following two subsets.
\begin{enumerate}
\item[\textup{(i)}] There exists
\[
  \widehat{\Xi}_{h,R}\subset\Xi_{h,R},
  \qquad
  \#\widehat{\Xi}_{h,R}\asymp h^{-d},
\]
such that, under the map $\Pi_R$ from \eqref{eq:scaled-parameter-map}, the set $\Pi_R(\widehat{\Xi}_{h,R})$ is quasiuniform at spherical scale $h$ on the strip $\Gamma_{C_b}$ from \eqref{eq:spherical-strip-definition}.

\item[\textup{(ii)}] There exists a fixed finite subset $P_{h,R}\subset\Xi_{h,R}$ such that every polynomial in $\mathcal P_k(\mathbb R^d)$ is reproduced exactly on $B_R$, with the uniform coefficient bounds in \autoref{lem:polynomial-ridge-reproduction}.
\end{enumerate}
Both conclusions are formulated entirely in terms of the ordinary $d_0$ fill distance and separation.
\end{lemma}

\begin{proof}
We first select the subset used for Sobolev approximation with fixed parameters. Set $\Theta_h:=\Pi_R(\Xi_{h,R})\subset\Gamma_{C_b}$. The map
\[
  (a,t)\mapsto\frac{(a,t)}{\sqrt{1+t^2}},
  \qquad (a,t)\in\mathbb S^{d-1}\times[-C_b,C_b],
\]
is a smooth embedding whose differential has uniform upper and lower bounds on this compact set. Hence,
\begin{equation}
\label{eq:strip-map-distance-comparison}
  \rho_{\mathbb S^d}(\Pi_R(a,b),\Pi_R(a',b'))
  \asymp
  \rho_{\mathbb S^{d-1}}(a,a')+
  \frac{|b-b'|}{1+R}.
\end{equation}
Fix $\eta=\Pi_R(a,b)\in\Gamma_{C_b}$. By the ordinary fill distance property, there exists $(a',b')\in\Xi_{h,R}$ such that $d_0((a,b),(a',b'))\lesssim h.$ Consequently, \eqref{eq:strip-map-distance-comparison} gives $\rho_{\mathbb S^d}(\eta,\Pi_R(a',b'))\lesssim h.$ Thus, $\Theta_h$ has $O(h)$ fill distance on $\Gamma_{C_b}$.

Choose a subset $\widehat\Theta_h$ of $\Theta_h$ that is maximal subject to pairwise distances of at least $ch$, with fixed $c>0$ sufficiently small. Maximality implies that its fill distance in $\Theta_h$ is less than $ch$, so it still has $O(h)$ fill distance on the entire strip. Since the strip is a fixed compact smooth manifold of dimension $d$, standard spherical packing and covering estimates give $\#\widehat\Theta_h\asymp h^{-d}.$ Taking the preimage under $\Pi_R$ gives $\widehat\Xi_{h,R}$ and proves \textup{(i)}.

For the polynomial stencil, let $\bar a_\ell,\beta_r$ be the fixed data from \autoref{lem:polynomial-ridge-reproduction} and define the target parameters $\bar\theta_{\ell r}:=(\bar a_\ell,-\beta_r(1+R)).$ The ordinary fill distance provides $\theta_{\ell r}=(a_{\ell r},b_{\ell r})\in\Xi_{h,R}$ such that $d_0(\theta_{\ell r},\bar\theta_{\ell r})\lesssim h.$ The target parameters have a positive mutual separation bounded below independently of $R$. Choose $h_0$ small enough that the selected points are pairwise distinct. Moreover,
\[
  \rho_{\mathbb S^{d-1}}(a_{\ell r},\bar a_\ell)\lesssim h,
  \qquad
  \left|-\frac{b_{\ell r}}{1+R}-\beta_r\right|\lesssim h.
\]
For sufficiently small $h_0$, these parameters satisfy the hypotheses of \autoref{lem:polynomial-ridge-reproduction}. Define
\[
  P_{h,R}
  :=\{\theta_{\ell r}:1\le\ell\le L,\ 0\le r\le k\}
  \subset\Xi_{h,R}.
\]
That lemma gives exact polynomial reproduction and the uniform coefficient bounds.
\end{proof}

To combine approximation with the inverse inequality, we next construct parameter sets whose coverage and antipodal separation have the same scale.

\begin{lemma}
\label{lem:shifted-product-parameters}
For $d\ge2$, let $A_h\subset\mathbb S^{d-1}$ be quasiuniform at scale $h$; for $d=1$, set $A_h=\mathbb S^0$. For every $R\ge1$ and sufficiently small $h>0$, there exists a set $B_{h,R}\subset[-C_b(1+R),C_b(1+R)]$ such that $\Xi_{h,R}:=A_h\times B_{h,R}$ satisfies
\begin{equation}
\label{eq:shifted-product-geometry}
  h(\Xi_{h,R})\asymp\delta(\Xi_{h,R})\asymp\delta_\pm(\Xi_{h,R})\asymp h,\qquad \#\Xi_{h,R}\asymp(1+R)h^{-d},
\end{equation}
with constants independent of $h,R$.
\end{lemma}

\begin{proof}
Set $n_b:=\lceil2C_b(1+R)/h\rceil$ and $\Delta_b:=2C_b(1+R)/n_b$, and define
\[
  B_{h,R}:=\{-C_b(1+R)+(m+3/4)\Delta_b:0\le m<n_b\}.
\]
For sufficiently small $h$, $ch\le\Delta_b\le h$, with $c>0$ independent of $R,h$. The distances from the bias grid to the two interval endpoints are $3\Delta_b/4$ and $\Delta_b/4$, while the largest interior gap is $\Delta_b$. Hence, both its ordinary fill distance and its separation are comparable to $h$. If $d\ge2$, the $d_0$ fill distance of the product set is controlled by the sum of the directional and bias fill distances, and its ordinary separation is controlled by the smaller of the two corresponding separations. Together with $\#A_h\asymp h^{-(d-1)}$ and $n_b\asymp(1+R)h^{-1}$, this gives the stated fill distance, ordinary separation, and cardinality bounds. If $d=1$, the directional set is the fixed sphere $\mathbb S^0$, consisting of two points, so the bias grid alone determines the small geometric scale and $\#A_h=2$ gives the same cardinality conclusion.

For $b_m,b_n\in B_{h,R}$, using $2C_b(1+R)=n_b\Delta_b$, $b_m+b_n=(m+n+3/2-n_b)\Delta_b.$ Since $m+n-n_b\in\mathbb Z$ and $\operatorname{dist}(3/2,\mathbb Z)=1/2$, $|b_m+b_n|\ge\frac12\Delta_b\gtrsim h$ for all $m,n.$ Thus, all $d_\pi$ distances have a uniform lower bound $ch$. Combining this with $\delta(\Xi_{h,R})\asymp h$ and $\delta_\pm\le\delta$ yields \eqref{eq:shifted-product-geometry}.
\end{proof}

\subsection{Deterministic sampling}
\label{subsec:fractional-sampling}

In this subsection, we establish deterministic sampling inequalities and the operator, tail, and stability estimates used in both sampling arguments.

For the sampling estimates, fix $c_0>0$ and let $0\le r\le S<k$ be integers with $k-r>(d-1)/2$. The order relations \eqref{eq:deterministic-order-identities} will be used in both sampling arguments.

The sampling conditions below are $h_X\lesssim H$ and $q_X\gtrsim H$, with implicit constants fixed independently of $H,R,N$. Choose a fixed $a>0$ sufficiently large that $aH>h_X$, and then fix $A>a$ sufficiently large in terms of $d,k-r$ and these constants. For a deterministic physical center set $X=\{x_i\}_{i=1}^N\subset B_R$, define
\begin{equation}
\label{eq:fractional-patches}
\begin{aligned}
  Q_i&:=B(x_i,aH)\cap B_R,
  &Q_i^*&:=B(x_i,AH)\cap B_R,\\
  I_i&:=\{j:x_j\in Q_i^*\},
  &X_i&:=\{x_j:j\in I_i\}.
\end{aligned}
\end{equation}
On all local patches $Q$ and their normalized copies in the deterministic and random sampling arguments, we equip $H^t(Q)$ with the interpolation norm of $[H^n(Q),H^{n+1}(Q)]_\theta$, where $t=n+\theta$, $n\in\mathbb N_0$, and $0<\theta<1$. Here $[\cdot,\cdot]_\theta$ denotes complex interpolation with the intrinsic integer derivative norms at the endpoints. For this patch family, set
\begin{equation}
\label{eq:fractional-patch-norm}
  \|g\|_{t,H}^2
  :=
  \sum_{i=1}^N\rho_d(x_i)
  \|g\|_{H^t(Q_i^*)}^2.
\end{equation}
For integer $t$, the same notation denotes the corresponding weighted direct sum of intrinsic $H^t(Q_i^*)$ norms. On the growing balls $B_R$, $R\ge1$, we retain the Bessel potential norm obtained by restriction from $\mathbb R^d$. Unless stated otherwise, constants depend only on the fixed geometric, regularity, weight, and operator data.

Let $q=n+\theta$, where $n\in\mathbb N_0$ and $0<\theta<1$. Write
\[
  \|v\|_{W^{q,2}(Q)}^2:=\sum_{|\beta|\le n}\|D^\beta v\|_{L^2(Q)}^2+\sum_{|\alpha|=n}[D^\alpha v]_{W^{\theta,2}(Q)}^2,
\]
where $[w]_{W^{\theta,2}(Q)}^2:=\int_Q\int_Q|w(y)-w(z)|^2|y-z|^{-d-2\theta}\,dy\,dz$.

The local sampling argument requires extension and norm comparison bounds that are uniform over the growing balls and the normalized patches.

\begin{lemma}
\label{lem:uniform-fractional-interfaces}
\label{lem:fractional-patch-structure}
\label{lem:fractional-interpolation-intrinsic-bridge}
\label{lem:fractional-gaussian-localization}
\label{lem:uniform-growing-ball-interpolation}
Let $Q=B_R$ or $Q=(AH)^{-1}((B(x,AH)\cap B_R)-x)$, with $R\ge1$, $x\in B_R$, and $0<AH\le R/4$. For every integer $L\ge1$, there is a common linear extension $E_Q$, with $(E_Qv)|_Q=v$, satisfying
\begin{equation}
\label{eq:uniform-sobolev-extension}
  \|E_Qv\|_{H^t(\mathbb R^d)}\le C\|v\|_{H^t(Q)},\qquad 0\le t\le L.
\end{equation}
The spaces $H^q(Q)$, $[H^n(Q),H^{n+1}(Q)]_\theta$, and $W^{q,2}(Q)$ coincide, with uniform norm equivalence
\begin{equation}
\label{eq:uniform-fractional-norm-equivalence}
  \|v\|_{H^q(Q)}\asymp\|v\|_{[H^n(Q),H^{n+1}(Q)]_\theta}\asymp\|v\|_{W^{q,2}(Q)}.
\end{equation}

If $q_X\gtrsim H$, $H(R+H)\le c_0$, and $AH\le R/4$, then for every noninteger $t\ge0$,
\begin{equation}
\label{eq:fractional-gaussian-localization}
  \|g\|_{t,H}\le C\|g\|_{\mathcal H_\gamma^t(\mathbb R^d)}.
\end{equation}
\end{lemma}

\begin{proof}
For $x\in B_R$ and $0<\rho\le R/4$, the intersection $B(x,\rho)\cap B_R$ is convex, has diameter at most $2\rho$, and contains a ball of radius $\rho/4$. Indeed, if $|x|\le R-\rho/2$, it contains $B(x,\rho/2)$. Otherwise, set $y:=x-\rho x/(2|x|)$. For $z\in B(y,\rho/4)$, one has $|z-x|<3\rho/4<\rho$ and $|z|<|x|-\rho/4<R$, so this ball lies in the intersection. Taking $\rho=AH$ and normalizing, each patch $Q$ is therefore a convex subset of $B_1$ containing a ball of radius $1/4$, with uniform chunkiness and interior cone constants.

On a fixed Lipschitz domain, the common extension at integer orders and the equivalence between the difference quotient norm and the real interpolation norm follow from \cite[Theorems~1.4.5 and~14.2.3]{BrennerScott2008}. For the normalized patches, the uniform Lipschitz charts and cutoffs supplied by this geometry give common extensions at integer orders with uniformly bounded norms by the same construction.

For any open convex $Q$, set $Q_h=\{y\in Q:y+h\in Q\}$. For $w_1\in H^1(Q)$, convexity allows integration along the segment from $y$ to $y+h$, giving $w_1(y+h)-w_1(y)=\int_0^1h\cdot\nabla w_1(y+sh)\,ds$ for almost every $y\in Q_h$. Hence $\|w(\cdot+h)-w\|_{L^2(Q_h)}\le2\|w_0\|_2+|h|\|\nabla w_1\|_2$ whenever $w=w_0+w_1$, with $w_0\in L^2(Q)$ and $w_1\in H^1(Q)$. Define
\[
  K(t,v):=\inf_{v=v_0+v_1}\bigl(\|v_0\|_{H^n(Q)}+t\|v_1\|_{H^{n+1}(Q)}\bigr).
\]
Applying the same decomposition of $v$ to every $D^\alpha v$, $|\alpha|=n$, gives $\sum_{|\alpha|=n}\|D^\alpha v(\cdot+h)-D^\alpha v\|_{L^2(Q_h)}^2\le CK(|h|,v)^2$. Integrating against $|h|^{-d-2\theta}\,dh$ and using polar coordinates, we obtain
\[
  \sum_{|\alpha|=n}[D^\alpha v]_{W^{\theta,2}(Q)}^2\le C\int_0^\infty t^{-1-2\theta}K(t,v)^2\,dt.
\]
The integral is the squared real interpolation norm. The spectral representation of a Hilbert pair and the weighted $L^2$ interpolation formulas in \cite[Theorems~5.4.1 and~5.5.3]{BerghLofstrom1976} identify it, up to constants depending only on $\theta$, with the squared complex interpolation norm. Consequently, for every open convex $Q$, independently of its size,
\begin{equation}
\label{eq:interpolation-top-derivative-bridge}
  \sum_{|\alpha|=n}[D^\alpha v]_{W^{\theta,2}(Q)}^2\le C\|v\|_{[H^n(Q),H^{n+1}(Q)]_\theta}^2.
\end{equation}
In particular, on normalized patches, $\|v\|_{W^{q,2}(Q)}\lesssim\|v\|_{H^q(Q)}$ uniformly.

For the reverse bound on a normalized patch $Q$, choose $\xi$ with $B(\xi,1/4)\subset Q$ and a fixed nonnegative $\eta\in C_c^\infty(B_1)$ with $\int\eta=1$. For $0<t<1/4$, set
\[
  S_tv(y):=\int_{B_1}\eta(z)v((1-t)y+t(\xi+z/4))\,dz.
\]
Convexity keeps all evaluation points in $Q$. Write $S_tw(y)=\int_Q\eta_t(y,z)w(z)\,dz$. The nonnegative kernel has unit mass, column integrals bounded by $(1-t)^{-d}$, support in $|y-z|\le3t$, and bounds $|\eta_t|\le Ct^{-d}$ and $|\nabla_y\eta_t|\le Ct^{-d-1}$. Since $\int_Q\nabla_y\eta_t(y,z)\,dz=0$,
\[
  \nabla S_tw(y)=\int_Q\nabla_y\eta_t(y,z)\bigl(w(z)-w(y)\bigr)\,dz.
\]
Together with $w(y)-S_tw(y)=\int_Q\eta_t(y,z)(w(y)-w(z))\,dz$, Jensen's and Cauchy--Schwarz inequalities give
\[
  \|w-S_tw\|_2^2+t^2\|\nabla S_tw\|_2^2\le Ct^{-d}\int_Q\int_{Q\cap B(y,3t)}|w(y)-w(z)|^2\,dz\,dy.
\]
For $w\in H^1(Q)$, integration along the same convex segments also gives $\|w-S_tw\|_2\le Ct\|\nabla w\|_2$. This controls derivatives below order $n$. For the top derivatives, use the preceding difference estimate and $D^\beta S_tv=(1-t)^{|\beta|}S_tD^\beta v$, noting that $|1-(1-t)^{|\beta|}|\le Ct$ for $|\beta|\le n$. The decomposition $v=(v-S_tv)+S_tv$ therefore gives
\[
  K(t,v)^2\le Ct^2\|v\|_{H^n(Q)}^2+Ct^{-d}\sum_{|\alpha|=n}\int_{Q}\int_{Q\cap B(y,3t)}|D^\alpha v(y)-D^\alpha v(z)|^2\,dz\,dy.
\]
After integration against $t^{-1-2\theta}\,dt$ on $(0,1/4)$, the first term is bounded because $\int_0^{1/4}t^{1-2\theta}\,dt<\infty$. For the second term, Tonelli's theorem and $\int_0^{1/4}t^{-d-1-2\theta}\mathbf1_{\{|y-z|<3t\}}\,dt\le C|y-z|^{-d-2\theta}$ give precisely the fractional seminorms. For $t\ge1/4$, use $K(t,v)\le\|v\|_{H^n(Q)}$. Combining the two ranges yields
\[
  \int_0^\infty t^{-1-2\theta}K(t,v)^2\,dt\le C\|v\|_{W^{q,2}(Q)}^2.
\]
This proves the uniform reverse bound. Interpolating the common extension at integer orders then gives \eqref{eq:uniform-sobolev-extension} on $Q$.

For $B_R$, set $E_{B_R}v(x):=E_{B_1}[v(R\,\cdot)](x/R)$, where $E_{B_1}$ is a common extension on the unit ball. The fixed-domain norm equivalence in \cite{BrennerScott2008} also makes $E_{B_1}$ bounded in $W^{q,2}$. Since $\|D^\beta[v(R\,\cdot)]\|_{L^2(B_1)}=R^{|\beta|-d/2}\|D^\beta v\|_{L^2(B_R)}$, its bounds at integer orders give, for each integer $m\ge0$ under consideration,
\[
  \sum_{|\alpha|=m}\|D^\alpha E_{B_R}v\|_{L^2(\mathbb R^d)}^2\le C\sum_{|\beta|\le m}R^{2(|\beta|-m)}\|D^\beta v\|_{L^2(B_R)}^2.
\]
Scaling the highest fractional seminorm similarly gives
\[
  \sum_{|\alpha|=n}[D^\alpha E_{B_R}v]_{W^{\theta,2}(\mathbb R^d)}^2\le C\sum_{|\beta|\le n}R^{2(|\beta|-q)}\|D^\beta v\|_{L^2(B_R)}^2+C\sum_{|\alpha|=n}[D^\alpha v]_{W^{\theta,2}(B_R)}^2.
\]
All powers of $R$ on the right are at most $1$ for $R\ge1$. Combining these estimates yields uniform integer and intrinsic fractional extension bounds. Interpolating the restriction map and this common extension at the integer endpoints identifies the interpolation norm with the norm obtained by restriction from $\mathbb R^d$, uniformly in $R$. This proves \eqref{eq:uniform-sobolev-extension} and \eqref{eq:uniform-fractional-norm-equivalence} for $B_R$.

Finally, separation gives bounded patch overlap. For $y\in Q_i^*$, $\left|\log\frac{\rho_d(y)}{\rho_d(x_i)}\right|\le R|y-x_i|+\frac12|y-x_i|^2\le AHR+\frac12A^2H^2\le Cc_0$, so
\begin{equation}
\label{eq:patch-density-comparison}
  \rho_d(y)\asymp\rho_d(x_i),\qquad y\in Q_i^*,
\end{equation}
uniformly. Consequently, for every integer $j\ge0$,
\[
  \sum_i\rho_d(x_i)\|g\|_{H^j(Q_i^*)}^2\le C\sum_{|\alpha|\le j}\int_{\mathbb R^d}|D^\alpha g|^2\,d\gamma\le C\|g\|_{\mathcal H_\gamma^j}^2,
\]
where the last inequality uses \eqref{eq:Hermite-derivative-norm-equivalence}. Apply complex interpolation to the same restriction map $g\mapsto(g|_{Q_i^*})_i$ at orders $j=n,n+1$. The Hermite coefficient representation gives $[\mathcal H_\gamma^n,\mathcal H_\gamma^{n+1}]_\theta=\mathcal H_\gamma^{n+\theta}$, while interpolation commutes with the weighted $\ell^2$ direct sums. The weights $\rho_d(x_i)$ are the same at both endpoints, and the interpolated patch spaces are precisely the spaces $H^{n+\theta}(Q_i^*)$ defined above. Thus $\|g\|_{n+\theta,H}\le C\|g\|_{\mathcal H_\gamma^{n+\theta}}$, proving \eqref{eq:fractional-gaussian-localization}.
\end{proof}

These uniform local estimates yield the two sampling bounds needed below. The upper bound controls sampled residuals of the approximant, while the lower bound recovers continuous norms from sampled values.

\begin{lemma}
\label{lem:deterministic-fractional-sampling}
\label{lem:fractional-scaled-point-evaluation}
\label{lem:fractional-local-fill}
\label{lem:fractional-polynomial-norming}
\label{lem:fractional-bramble-hilbert}
\label{lem:fractional-empirical-upper}
\label{lem:fractional-local-sampling}
Choose $A>\max\{a,1\}$ sufficiently large in terms of $d,k-r$ and the sampling constants, and $H_*>0$ with $AH_*\le1$. Let $R\ge1$ and $0<H\le H_*$ satisfy $h_X\lesssim H$, $q_X\gtrsim H$, $H(R+H)\le c_0$, and $AH\le R/4$.

Every $g\in L^2(B_R)$ with $\|g\|_{k-r+1/4,H}<\infty$ has a unique continuous representative on $B_R$. If the upper weight bound in \eqref{eq:weight-condition} holds, then
\begin{equation}
\label{eq:fractional-empirical-upper}
  \|g\|_{X,\omega}
  \le C\left[
    \|g\|_{L_\gamma^2(B_R)}
    +H^{k-r+1/4}\|g\|_{k-r+1/4,H}
  \right].
\end{equation}
Under both weight bounds in \eqref{eq:weight-condition},
\begin{equation}
\label{eq:fractional-local-sampling}
\|g\|_{H_\gamma^{S-r}(B_R)} \le C\bigl[ H^{k-S+1/4}\|g\|_{k-r+1/4,H} +H^{-(S-r)}\|g\|_{X,\omega} \bigr].
\end{equation}
\end{lemma}

\begin{proof}
Let $\widehat Q_i^*=(AH)^{-1}(Q_i^*-x_i)$ and put $v(\xi)=g(x_i+AH\xi)$. Using the extension in \eqref{eq:uniform-sobolev-extension}, set $F:=E_{\widehat Q_i^*}v$. Fourier splitting gives $\|F\|_{H^{k-r}}\le\varepsilon\|F\|_{H^{k-r+1/4}} +C_\varepsilon\|F\|_2$; after restriction, \eqref{eq:uniform-fractional-norm-equivalence}, the common $L^2$ endpoint bound, and absorption imply
\[
 \|v\|_{H^{k-r}(\widehat Q_i^*)}
 \lesssim \|v\|_2+
 \Bigl(\sum_{|\alpha|=k-r}[D^\alpha v]_{W^{1/4,2}}^2\Bigr)^{1/2}.
\]
Together with \eqref{eq:uniform-sobolev-extension}--\eqref{eq:uniform-fractional-norm-equivalence} and $H^{k-r+1/4}(\mathbb R^d)\hookrightarrow C^0$, this gives the normalized supremum estimate and a continuous representative. The change of variables yields
\[
  \|v\|_2^2=(AH)^{-d}\|g\|_2^2,\qquad [D^\alpha v]_{W^{1/4,2}}^2=(AH)^{2(k-r)+1/2-d}[D^\alpha g]_{W^{1/4,2}}^2.
\]
Using \eqref{eq:interpolation-top-derivative-bridge}, we obtain, for every $g\in H^{k-r+1/4}(Q_i^*)$,
\begin{equation}
\label{eq:fractional-scaled-supremum-estimate}
  \|g\|_{L^\infty(Q_i^*)}^2
  \le C\left[
    H^{-d}\|g\|_{L^2(Q_i^*)}^2
    +H^{2(k-r)+1/2-d}\|g\|_{H^{k-r+1/4}(Q_i^*)}^2
  \right].
\end{equation}

The patch representatives agree on overlaps as continuous versions of the same $L^2$ function, and $AH>h_X$ makes the patches cover $B_R$. Multiply \eqref{eq:fractional-scaled-supremum-estimate} by $\omega_i$ and sum over $i$. The upper bound in \eqref{eq:weight-condition}, the density comparison \eqref{eq:patch-density-comparison}, and the bounded overlap of the $Q_i^*$ prove \eqref{eq:fractional-empirical-upper}.

To prove \eqref{eq:fractional-local-sampling}, first work on a normalized patch $Q$ and define $P_{Q}v\in\mathcal P_{k-r}$ by
\begin{equation*}
  \int_{Q}D^\beta(v-P_{Q}v)=0,
  \qquad |\beta|\le k-r.
\end{equation*}
The moment system is triangular in decreasing total degree, with diagonal $\beta!|Q|>0$, so it defines a linear projector. For $w=v-P_Qv$, every derivative $D^\beta w$ with $|\beta|\le k-r$ has zero mean. Since $\operatorname{diam}Q\le2$ and $|Q|\ge c>0$, for $|\alpha|=k-r$ this gives
\[
  \|D^\alpha w\|_{L^2(Q)}^2=\frac{1}{2|Q|}\int_Q\int_Q|D^\alpha w(x)-D^\alpha w(y)|^2\,dx\,dy\le C[D^\alpha v]_{W^{1/4,2}(Q)}^2,
\]
where $D^\alpha P_Qv$ is constant. For $|\beta|<k-r$, use $D^\beta w(x)-D^\beta w(y)=\int_0^1(x-y)\cdot\nabla D^\beta w(y+t(x-y))\,dt$, initially for smooth functions and then by density. Cauchy--Schwarz and $\int_Q D^\beta w=0$ give
\[
  \|D^\beta w\|_{L^2(Q)}^2\le2^{d-1}(\operatorname{diam}Q)^2\|\nabla D^\beta w\|_{L^2(Q)}^2.
\]
Indeed, split the integral in $t$ at $1/2$ and change variables in $y$ on the first half and in $x$ on the second; convexity keeps both images in $Q$, and the inverse Jacobians are bounded by $2^d$. Iterating over the derivative orders yields
\[
 \|w\|_{H^{k-r}(Q)}
 \lesssim\Bigl(\sum_{|\alpha|=k-r}
 [D^\alpha v]_{W^{1/4,2}(Q)}^2\Bigr)^{1/2}.
\]
Using \eqref{eq:uniform-sobolev-extension}--\eqref{eq:uniform-fractional-norm-equivalence} and Sobolev embedding, we control the $H^{k-r+1/4}$ and $L^\infty$ remainders by the fractional seminorms of the derivatives of order $k-r$. Scaling the linear projector back to $Q_i^*$ gives a single polynomial $P_i\in\mathcal P_{k-r}$, linear in $g$. For every $0\le j\le k-r$, the scaling factors are $(AH)^{k-r+1/4-j}$ in $H^j$ and $(AH)^{k-r+1/4-d/2}$ in $L^\infty$; the lower derivative orders are absorbed because $AH\le1$. Applying \eqref{eq:interpolation-top-derivative-bridge} on $Q_i^*$ and absorbing the fixed powers of $A$ into $C$ gives
\begin{align}
\label{eq:fractional-BH-scaled-Hj}
  \|g-P_i\|_{H^j(Q_i^*)}
  &\le CH^{k-r-j+1/4}\|g\|_{H^{k-r+1/4}(Q_i^*)},\\
\label{eq:fractional-BH-scaled-Linf}
  \|g-P_i\|_{L^\infty(Q_i^*)}
  &\le CH^{k-r-d/2+1/4}\|g\|_{H^{k-r+1/4}(Q_i^*)}.
\end{align}
To control $P_i$ by the sampled values, extend the global fill bound to $\overline B_R$. The choice of $A$ ensures $AH>4h_X$. A target within $AH-2h_X$ of $x_i$ has a global fill point inside $Q_i^*$; otherwise move it $2h_X$ toward $x_i$ and take a fill point there. Convexity keeps the moved point in $B_R$, and the resulting distance is at most $3h_X\lesssim H$. Thus
\begin{equation}
\label{eq:fractional-complete-local-fill}
  h_{X_i,\overline{Q_i^*}}
  :=\sup_{y\in\overline{Q_i^*}}
  \min_{x_j\in X_i}|y-x_j|
  \lesssim H.
\end{equation}

Fix $C_*>0$, depending only on the sampling constants, large enough that $h_{X_i,\overline{Q_i^*}}\le C_*H$. For any $Y_i\subset\overline{Q_i^*}$ with fill distance at most $C_*H$, scaling by $H^{-1}$ gives a patch with uniform cone angle and cone radius $\gtrsim A$. Taking $A$ sufficiently large in terms of $d,k-r,C_*$, \cite[Theorem~3.8]{Wendland2005} yields
\begin{equation}
\label{eq:fractional-polynomial-norming}
  \|p\|_{L^\infty(Q_i^*)}
  \le2\max_{y\in Y_i}|p(y)|,
  \qquad p\in\mathcal P_{k-r}(\mathbb R^d).
\end{equation}
This applies to $X_i$ and to the representative point sets used in the random sampling proof.

Since each normalized patch $Q\subset B_1$ contains a ball of radius $1/4$, polynomials $p\in\mathcal P_{k-r}$ satisfy $\|p\|_{H^j(Q)}\le C\|p\|_{L^\infty(Q)}$ uniformly for $0\le j\le k-r$. Scaling back and using $Q_i\subset Q_i^*$ gives, for $p\in\mathcal P_{k-r}$,
\begin{equation}
\label{eq:fractional-uniform-polynomial-inverse}
  \|p\|_{H^j(Q_i)}
  \le CH^{d/2-j}\|p\|_{L^\infty(Q_i^*)},
  \qquad 0\le j\le k-r.
\end{equation}

Combining \eqref{eq:fractional-BH-scaled-Hj}--\eqref{eq:fractional-BH-scaled-Linf}, \eqref{eq:fractional-polynomial-norming}, and \eqref{eq:fractional-uniform-polynomial-inverse} for the same $P_i$, summing $0\le j\le S-r$, and using $H\le1$ gives
\begin{equation*}
  \|g\|_{H^{S-r}(Q_i)}^2
  \le C\left[
    H^{2(k-S)+1/2}\|g\|_{H^{k-r+1/4}(Q_i^*)}^2
    +H^{d-2(S-r)}\sum_{\ell\in I_i}|g(x_\ell)|^2
  \right].
\end{equation*}
Since the sets $Q_i$ cover $B_R$ and each point belongs to at most a fixed number of sets $Q_i^*$, summing the preceding inequality with weights $\rho_d(x_i)$ and using \eqref{eq:patch-density-comparison} and the lower bound in \eqref{eq:weight-condition} proves \eqref{eq:fractional-local-sampling}.
\end{proof}

To apply these sampling bounds to PDE residuals, we next estimate the local Sobolev norm of $\mathcal Av$ in terms of $v$.

\begin{lemma}
\label{lem:fractional-residual-calculus}
\label{lem:fractional-local-operator}
\label{lem:fractional-patchwise-continuity}
Assume the coefficient bounds in \autoref{ass:solution-local-growth}. Let $R\ge1$ and suppose $q_X\gtrsim H$, $H(R+H)\le c_0$, and $AH\le R/4$.
For $v\in\mathcal H_\gamma^{k+1/4}(\mathbb R^d)$ and $v\in H^{k+1/4}(B_R)$, respectively,
\begin{align}
\label{eq:fractional-operator-global}
 \|\mathcal Av\|_{k-r+1/4,H}
 &\le C(1+R)^{p_2}\|v\|_{\mathcal H_\gamma^{k+1/4}(\mathbb R^d)},\\
\label{eq:fractional-operator-local}
 \|\mathcal Av\|_{k-r+1/4,H}
 &\le C(1+R)^{p_2}\|v\|_{H^{k+1/4}(B_R)}.
\end{align}
If also $h_X\lesssim H$ and $AH>h_X$, these residuals have unique continuous representatives on $B_R$.
\end{lemma}

\begin{proof}
At $j=k-r$ and $j=k-r+1$, Leibniz's rule gives
\begin{equation*}
 \|\mathcal Av\|_{H^j(Q_i^*)}
 \le C(1+R)^{p_2}\|v\|_{H^{r+j}(Q_i^*)}.
\end{equation*}
The endpoint $k-r+1$ uses the highest coefficient derivative required in \autoref{ass:solution-local-growth}. Squaring, weighting by $\rho_d(x_i)$, summing, and interpolating the same operator between the two weighted Hilbert direct sums gives
\begin{equation}
\label{eq:fractional-operator-patch}
 \|\mathcal Av\|_{k-r+1/4,H}\le C(1+R)^{p_2}\|v\|_{k+1/4,H}.
\end{equation}
Combining this with \eqref{eq:fractional-gaussian-localization} at $k+1/4$ proves \eqref{eq:fractional-operator-global}.

For local data, bounded overlap gives uniform endpoint restriction maps $H^{r+j}(B_R)\to(\bigoplus_iH^{r+j}(Q_i^*))_{\ell^2(\rho_d(x_i))}$. Take the infimum over extensions to $\mathbb R^d$ and interpolate, using \eqref{eq:uniform-fractional-norm-equivalence} with $n=k$ and $\theta=1/4$, to obtain \eqref{eq:fractional-operator-local}.  Finally, the endpoint estimates identify one $L^2_{\rm loc}$ residual; fractional extension and $k-r+1/4>d/2$ make its patchwise versions continuous. They coincide on overlaps, and the fill condition makes the patches cover $B_R$, so the residual has a unique continuous representative on $B_R$.
\end{proof}

To obtain global stability from residual samples in $B_R$, we also need to control the Gaussian tails.

\begin{lemma}
\label{lem:coefficient-tail}
\label{lem:optimal-approximation-tail}
\label{lem:tail-from-local-growth}
The following bounds hold for every fixed $0<c<1/4$, with constants $C$ allowed to depend on $c$.
\begin{enumerate}
\item[\textup{(i)}] Let $r,s\ge0$ be integers with $r+s\le k$, let $v=\sum_{j=1}^Jc_j\phi_{\theta_j}$ with $\theta_j\in\Theta_R$ and $\sum_j|c_j|\le M$, and suppose $\cA=\sum_{|\alpha|\le r}a_\alpha D^\alpha$ satisfies $|D^\beta a_\alpha(x)|\le C(1+|x|)^{p_2}$ a.e.\ for $|\alpha|\le r$, $|\beta|\le s$. Then
\[
 \|\cA v\|_{H_\gamma^s(B_R^c)}\le CM(1+R)^{p_2+k}e^{-cR^2}.
\]
The constants are independent of $R,J,M$ and the parameter arrangement.

\item[\textup{(ii)}] Let $w_h$ be the approximant from \autoref{lem:convergence-approximant}. For every integer $0\le q\le k$,
\begin{equation*}
 \|w_h\|_{H_\gamma^q(B_R^c)}\le C(1+R)^{k+1/2}e^{-cR^2}\|u\|_{H^{s_k}(B_R)}.
\end{equation*}

\item[\textup{(iii)}] Under \autoref{ass:solution-local-growth}, every integer $0\le q<s_k$ satisfies
\begin{equation}
\label{eq:solution-tail-general-order}
 \|u\|_{H_\gamma^q(B_R^c)}\le C(1+R)^{p_1}e^{-cR^2},\qquad R\ge1.
\end{equation}
\end{enumerate}
\end{lemma}

\begin{proof}
For (i), each derivative $D^\beta\cA\phi_{\theta_j}$ with $|\beta|\le s$ is bounded by $C(1+|x|)^{p_2+k}$. For any fixed $0<c<1/4$, polar coordinates give $\int_{|x|>R}(1+|x|)^{2(p_2+k)}\,d\gamma(x)\le C_c(1+R)^{2(p_2+k)}e^{-2cR^2}$: the remaining radial power is absorbed by $e^{-(1/2-2c)R^2}$. The triangle inequality and $\sum_j|c_j|\le M$ give the stated estimate. Part (ii) follows from (i) applied to the identity operator and the coefficient bound \eqref{eq:optimal-coefficient-bound}.

For (iii), decompose $B_R^c$ into the annuli $R+\ell<|x|\le R+\ell+1$, $\ell\ge0$. On each annulus, the Gaussian density is at most $Ce^{-(R+\ell)^2/2}$. Thus, for every fixed $0<c<1/4$,
\[
 \|u\|_{H_\gamma^q(B_R^c)}^2\le C\sum_{\ell\ge0}e^{-(R+\ell)^2/2}\|u\|_{H^{s_k}(B_{R+\ell+1})}^2\le C_c(1+R)^{2p_1}e^{-2cR^2}.
\]
Taking square roots proves \eqref{eq:solution-tail-general-order}.
\end{proof}

The inverse inequality in \autoref{thm:inverse} lets us absorb the remainder in the lower sampling bound and obtain discrete stability. We use a common seminorm $D_{R,H}$, taking $\|\cdot\|_{X,\omega}$ for deterministic sampling and $\|\cdot\|_{N,\gamma}$ on the event where the random sampling estimates hold.

\begin{lemma}
\label{lem:fractional-admissible-data-transfer}
\label{lem:fractional-data-norm-transfer}
\label{lem:fractional-constrained-stability}
Assume the geometric conditions of \autoref{lem:deterministic-fractional-sampling}, the coefficient bounds in \autoref{ass:solution-local-growth}, and the stability estimate for network functions in \autoref{ass:apriori}. Suppose that \eqref{eq:fractional-local-sampling} holds with $\|g\|_{X,\omega}$ replaced by $D_{R,H}(g)$, for the same class of functions and with a constant independent of $R,H,g$.

Let $\Xi\subset\Theta_R$ be finite with $\delta_\pm(\Xi)>0$, let $L>0$, and fix $0<c<1/4$. There is a fixed $C>0$, allowed to depend on $c$, such that, if $C(1+R)^{p_2}\bigl(H\max\{\delta_\pm(\Xi)^{-1},1+R\}\bigr)^{k-S+1/4}\le1/2$, every $z\in V_\Xi(L)$ satisfies
\begin{equation}
\label{eq:fractional-constrained-stability}
  \|z\|_{\mathcal H_\gamma^S}\le CH^{-(S-r)}D_{R,H}(\mathcal Az)+CL(1+R)^{p_2+k}e^{-cR^2}.
\end{equation}
\end{lemma}

\begin{proof}
The stability estimate in \autoref{ass:apriori}, \eqref{eq:Hermite-derivative-norm-equivalence}, \eqref{eq:fractional-local-sampling} with $D_{R,H}$ in place of $\|\cdot\|_{X,\omega}$, and \eqref{eq:fractional-operator-global} give
\[
\begin{aligned}
 \|z\|_{\mathcal H_\gamma^S}
 &\le C(1+R)^{p_2}
 \bigl(H\max\{\delta_\pm(\Xi)^{-1},1+R\}\bigr)^{k-S+1/4}\|z\|_{\mathcal H_\gamma^S}\\
 &\quad+CH^{-(S-r)}D_{R,H}(\mathcal Az)
 +C\|\mathcal Az\|_{H_\gamma^{S-r}(B_R^c)}.
\end{aligned}
\]
Here \autoref{thm:inverse} is used at $\sigma=S$ and $\tau=k+1/4<k+1/2$. Choosing the constant in the condition on $H$ sufficiently large makes the first term at most $\frac12\|z\|_{\mathcal H_\gamma^S}$. Moreover, $S-r\le k-r-1$ makes the coefficient bounds in \autoref{ass:solution-local-growth} sufficient at the exterior integer order; applying them on countably many balls gives the global a.e.\ growth hypothesis, and $S<k$ lets \autoref{lem:coefficient-tail} bound the exterior term by $CL(1+R)^{p_2+k}e^{-cR^2}$. Absorption proves \eqref{eq:fractional-constrained-stability}.
\end{proof}

\subsection{Random sampling}
\label{subsec:random-physical-sampling-proof}

In this subsection, we extend the sampling estimates to independent uniform points.

The argument fixes a deterministic partition, smaller balls inside its cells, and separated patch centers before drawing the independent uniform samples. Every observation enters the norm $\|\cdot\|_{N,\gamma}$, while choosing one point from each smaller ball gives the lower sampling bound.

\begin{lemma}
\label{lem:random-partition-occupancy}
\label{lem:random-cell-partition}
\label{lem:balanced-cell-occupancy}
Let $R\ge1$ and $0<H\le R$. Before sampling one can fix a convex Borel partition $B_R=\bigcup_{\ell=1}^{M_0}V_\ell$ with null boundaries, centers $Z_{H,R}=\{z_\ell\}_{\ell=1}^{M_0}\subset B_R$, points $\xi_\ell$, and balls $W_\ell:=B(\xi_\ell,H/384)\subset V_\ell$ such that
\begin{equation}
\label{eq:random-cell-geometry}
 M_0\asymp(R/H)^d,\qquad B(\xi_\ell,H/96)\subset V_\ell\subset B(z_\ell,H/8),\qquad |V_\ell|\asymp|W_\ell|\asymp H^d.
\end{equation}
For independent uniform samples $Y_1,\ldots,Y_N$ in $B_R$, let $\mathcal E_{R,H,N}$ be the event that every $W_\ell$ contains at least $N|W_\ell|/(2|B_R|)$ samples and every $V_\ell$ contains at most $3N|V_\ell|/(2|B_R|)$ samples.
There are $c,C>0$, depending only on $d$, such that
\begin{equation}
\label{eq:balanced-occupancy-probability}
 \mathbb P(\mathcal E_{R,H,N}^c)\le2M_0\exp(-cN/M_0).
\end{equation}
Consequently,
\begin{equation}
\label{eq:balanced-occupancy-sufficient}
 N\ge CM_0\log(CM_0/\eta)
\end{equation}
implies $\mathbb P(\mathcal E_{R,H,N})\ge1-\eta$.
\end{lemma}

\begin{proof}
Take a set $Z_{H,R}\subset B_R$ that is maximal subject to pairwise distances of at least $H/8$. Thus $|z_\ell-z_j|\ge H/8$ for $\ell\ne j$ and $q_{Z_{H,R}}\ge H/16$, while maximality gives $h_{Z_{H,R},B_R}\le H/8$. Disjoint balls of radius $H/16$ and volume comparison give the cell count in \eqref{eq:random-cell-geometry}. Separation also gives
\begin{equation}
\label{eq:random-cell-enlargement-overlap}
 \sup_{x\in B_R}\#\{\ell:x\in B(z_\ell,AH)\cap B_R\}\le C_{d,A}\qquad(A\ge1\ \hbox{fixed}).
\end{equation}
Using a fixed ordering to resolve ties, define the relative Voronoi cells
\[
 V_\ell:=B_R\cap\!\bigcap_{j<\ell}\{|x-z_\ell|<|x-z_j|\}\cap\!\bigcap_{j>\ell}\{|x-z_\ell|\le|x-z_j|\}.
\]
They are convex, form a Borel partition, and have null boundaries. Maximality gives $V_\ell\subset B(z_\ell,H/8)$, hence $\operatorname{diam}V_\ell\le H/4$, while separation gives $B(z_\ell,H/24)\cap B_R\subset V_\ell$. The geometric observation in the proof of \autoref{lem:uniform-fractional-interfaces}, with $\rho=H/24$, gives $B(\xi_\ell,H/96)\subset V_\ell$. These balls lie in disjoint cells, so
\begin{equation}
\label{eq:random-core-center-separation}
 |\xi_\ell-\xi_j|\ge H/48\qquad(\ell\ne j).
\end{equation}
The inclusions and the definition of $W_\ell$ give the volume bounds in \eqref{eq:random-cell-geometry}.

Set
\[
 K_\ell:=\#\{i:Y_i\in V_\ell\},\qquad L_\ell:=\#\{i:Y_i\in W_\ell\}.
\]
The volume bounds in \eqref{eq:random-cell-geometry} give $|W_\ell|/|B_R|\asymp|V_\ell|/|B_R|\asymp M_0^{-1}$. Multiplicative binomial lower and upper tails for $L_\ell$ and $K_\ell$, followed by a union bound over $\ell$, prove \eqref{eq:balanced-occupancy-probability}. Condition \eqref{eq:balanced-occupancy-sufficient} makes the failure probability at most $\eta$.
\end{proof}

We next use the cell counts on $\mathcal E_{R,H,N}$ to transfer the deterministic sampling estimates to $\|\cdot\|_{N,\gamma}$.

Choose fixed patch factors $a>1/4$ and $A$ as in \autoref{lem:fractional-local-sampling} for the bounds $h_X\le H/4$ and $q_X\ge H/128$, which hold for both the fixed centers and the core representatives below. Fix $A_*>A+1/8$.

For the rest of this subsection, use the patches and norm in \eqref{eq:fractional-patches}--\eqref{eq:fractional-patch-norm} with centers $Z_{H,R}$, $M_0$ patches, and outer factor $A_*$; $N$ continues to denote the number of random observations. The applications of \autoref{lem:fractional-residual-calculus} and \autoref{lem:fractional-data-norm-transfer} use these patches, while the application of \autoref{lem:fractional-local-sampling} centered at the core representatives uses the smaller factor $A$.

\begin{lemma}
\label{lem:random-full-sample-fractional-interface}
\label{lem:random-fractional-sampling-interface}
\label{lem:random-fractional-cell-representative-sampling}
\label{lem:random-fractional-empirical-upper}
\label{lem:random-fractional-physical-sampling}
Let $0\le r\le S<k$ be integers and assume $k-r>(d-1)/2$. There exist fixed $H_*,C>0$ such that, whenever $0<H\le H_*$, $H(R+H)\le c_0$, and $A_*H\le R/4$, every $g\in L^2(B_R)$ with $\|g\|_{k-r+1/4,H}<\infty$ has a unique continuous representative on $B_R$. On $\mathcal E_{R,H,N}$, simultaneously for every such $g$,
\begin{align}
\label{eq:random-fractional-empirical-upper}
 \|g\|_{N,\gamma}&\le C\left[\|g\|_{L_\gamma^2(B_R)}+H^{k-r+1/4}\|g\|_{k-r+1/4,H}\right],\\
\label{eq:random-fractional-physical-sampling}
 \|g\|_{H_\gamma^{S-r}(B_R)}&\le C\left[H^{k-S+1/4}\|g\|_{k-r+1/4,H}+H^{-(S-r)}\|g\|_{N,\gamma}\right].
\end{align}
Both estimates hold on the same event for every such $g$, including functions depending on the samples, and use all $N$ observations. Their constants are independent of $R,H,N,g$ and of the realization on the event.
\end{lemma}

\begin{proof}
Choose $H_*$ with $A_*H_*\le1$. The $H/8$ fill bound and $a>1/4$ make the inner patches cover $B_R$. The extension estimate \eqref{eq:uniform-sobolev-extension} and $k-r+1/4>d/2$ give continuous patch representatives. These representatives agree on overlaps and define a continuous function on $B_R$.

Fix $\omega\in\mathcal E_{R,H,N}$. Since $K_\ell\le3N|V_\ell|/(2|B_R|)$, applying \eqref{eq:fractional-scaled-supremum-estimate} and \eqref{eq:random-cell-enlargement-overlap} with outer factor $A_*$, together with \eqref{eq:random-cell-geometry} and \eqref{eq:patch-density-comparison}, gives
\[
 \|g\|_{N,\gamma}^2\le C\sum_\ell|V_\ell|\sup_{V_\ell}\rho_d|g|^2\le C\bigl(\|g\|_{L_\gamma^2(B_R)}^2+H^{2(k-r)+1/2}\|g\|_{k-r+1/4,H}^2\bigr).
\]
Taking square roots proves \eqref{eq:random-fractional-empirical-upper}.

To prove \eqref{eq:random-fractional-physical-sampling}, choose arbitrary $y_\ell\in W_\ell$ and set $X:=\{y_\ell\}_{\ell=1}^{M_0}$. Each $V_\ell$ lies within $H/4$ of $y_\ell$, while $|y_\ell-y_j|\ge H/48-2H/384=H/64$. Hence
\begin{equation}
\label{eq:random-core-representative-geometry}
 h_{X,B_R}\le H/4,\qquad q_X\ge H/128.
\end{equation}
By the volume bounds in \eqref{eq:random-cell-geometry}, the weights $|W_\ell|\rho_d(y_\ell)$ satisfy \eqref{eq:weight-condition}. The patches $\widetilde Q_\ell^*:=B(y_\ell,AH)\cap B_R$ satisfy $\widetilde Q_\ell^*\subset Q_\ell^*$ because $A_*>A+1/8$. Interpolating restriction from $Q_\ell^*$ to $\widetilde Q_\ell^*$ between integer orders and applying \eqref{eq:patch-density-comparison} with center $z_\ell$ and outer factor $A_*$ yields
\[
 \sum_\ell\rho_d(y_\ell)\|g\|_{H^{k-r+1/4}(\widetilde Q_\ell^*)}^2\lesssim\|g\|_{k-r+1/4,H}^2.
\]
Applying \autoref{lem:fractional-local-sampling} gives
\begin{equation}
\label{eq:random-fractional-cell-representative-sampling}
 \|g\|_{H_\gamma^{S-r}(B_R)}\le C\biggl[H^{k-S+1/4}\|g\|_{k-r+1/4,H}+H^{-(S-r)}\Bigl(\sum_{\ell=1}^{M_0}|W_\ell|\rho_d(y_\ell)|g(y_\ell)|^2\Bigr)^{1/2}\biggr].
\end{equation}
The constant is independent of $R,H,M_0$, the partition, and the chosen points $y_\ell$.

Every $W_\ell$ contains a sample on this event. Square \eqref{eq:random-fractional-cell-representative-sampling} and average over all choices of one sample index in each $W_\ell$, with $y_\ell$ equal to the corresponding observation. The resulting discrete term satisfies
\[
 \sum_\ell\frac{|W_\ell|}{L_\ell}\sum_{Y_i\in W_\ell}\rho_d(Y_i)|g(Y_i)|^2\le2\frac{|B_R|}{N}\sum_{i=1}^N\rho_d(Y_i)|g(Y_i)|^2=2\|g\|_{N,\gamma}^2.
\]
Taking square roots proves \eqref{eq:random-fractional-physical-sampling}. This finite averaging uses only the lower core counts, so the estimate holds simultaneously for all admissible $g$ on the same event, in the norm $\|\cdot\|_{N,\gamma}$.
\end{proof}

{\small
\bibliographystyle{plain}
\bibliography{gaussian_hermite_arxiv_compact_48pp_references}
}
\end{document}